\documentclass[a4paper]{scrartcl}
\usepackage{amsmath}
\usepackage{amsthm}
\usepackage{amssymb}
\usepackage{xypic}
\usepackage{graphicx}
\usepackage[utf8]{inputenc}

\newtheorem{theorem}{Theorem}[section]
\newtheorem{lemma}[theorem]{Lemma}
\newtheorem{proposition}[theorem]{Proposition}

\newtheorem{cor}[theorem]{Corollary}

\theoremstyle{definition}
\newtheorem{definition}[theorem]{Definition}
\newtheorem{example}[theorem]{Example}

\theoremstyle{remark}
\newtheorem{remark}[theorem]{Remark}

\numberwithin{equation}{section}
\newcommand{\Q}{\mathbb{Q}}
\newcommand{\C}{\mathbb{C}}

\newcommand{\R}{\mathbb{R}}

\DeclareMathOperator{\id}{Id}
\DeclareMathOperator{\Aut}{Aut}

\title{Non-negatively curved GKM orbifolds}

\author{Oliver Goertsches and Michael Wiemeler\thanks{MW was supported by DFG-Grants HA 3160/6-1 and HA 3160/11-1 and SFB 878.}}

\date{ }

\begin{document}

\maketitle

\begin{abstract}
  In this paper we study non-negatively curved and rationally elliptic GKM$_4$ manifolds and orbifolds. We show that their rational cohomology rings are isomorphic to the rational cohomology of certain model orbifolds.
These models are quotients of isometric actions of finite groups on non-negatively curved torus orbifolds.

Moreover, we give a simplified proof of a characterisation of products of simplices among orbit spaces of locally standard torus manifolds.
This characterisation was originally proved in \cite{MR3355120} and was used there to obtain a classification of non-negatively curved torus manifolds.
\end{abstract}


\section{Introduction}
\label{sec:introduction-2}

The classification of Riemannian manifolds with positive or non-negative sectional curvature is one of the most prominent open problems in differential geometry. Many authors have investigated it under the additional assumption of an isometric Lie group action, e.g.\ \cite{HsiangKleiner}, \cite{SearleYang}, \cite{GroveSearle}, \cite{Wilking}, \cite{FangRong}, \cite{AmannKennard},  \cite{MR2784821}, \cite{MR3355120},  \cite{GGKRW}, \cite{MR3897041}, \cite{MR4088352}. In this paper we continue our investigation \cite{GoertschesWiemeler} of isometric torus actions of GKM type on Riemannian manifolds with sectional curvature bounded from below. 


The GKM condition --- introduced in \cite{GKM} ---  requires that the one-skeleton \(M_1\) of the action, i.e. the union of all orbits of dimension less than or equal to one, is of a particularly simple type.
Namely, it is required that \(M_1\) is a union of two-dimensional spheres, such that the \(T\)-action restricts to a cohomogeneity one action on each two-sphere.
The orbit space \(\Gamma=M_1/T\) is then an \(n\)-valent graph, where \(2n=\dim M\). The isotropy representations at the fixed points induce a labeling of the edges of the graph, as explained in Section \ref{Preliminaries}.
From this labelled graph one can compute the equivariant and non-equivariant rational cohomology rings of a GKM manifold or orbifold.
This is made explicit in the GKM Theorem \cite{GKM}, see Theorem \ref{thm:gkm} below.

Similarly to the GKM condition, we say that an action is GKM$_k$ if for all $0\leq k'<k$ the union of the orbits of dimension at most $k'$ is a union of $2k'$-dimensional invariant submanifolds. Their GKM graphs are then the \(k'\)-dimensional faces of the GKM graph of \(M\). 

In \cite{GoertschesWiemeler} we showed that a positively curved Riemannian manifold admitting an isometric GKM$_3$ torus action has the same real cohomology ring as a compact rank one symmetric space. The assumption of positive curvature forces, by the classification of $4$-dimensional positively curved $T^2$-manifolds \cite{GroveSearle}, the two-dimensional faces of the GKM graph to be just biangles or triangles -- this condition turned out to be a severe enough restriction to classify all occurring graphs.

Considering the same setting for non-negatively instead of positively curved manifolds, we observe that now also quadrangles appear as two-dimensional faces   \cite{HsiangKleiner}, \cite{SearleYang}, which increases the possibilities for the GKM graphs greatly. Still, we are able to show the following theorem on the structure of the GKM graph (without the labelling):

\begin{theorem}
\label{sec:introduction}
  Let \(O\) be a orientable GKM$_4$ orbifold with an invariant metric of non-negative curvature. Then the GKM graph of \(O\) is finitely covered by the vertex-edge graph of a finite product \(\prod_i\Delta^{n_i}\times \prod_i\Sigma^{m_i}\).
\end{theorem}

In the above theorem and later on, \(\Delta^n\) denotes an \(n\)-dimensional simplex and \(\Sigma^m\) the orbit space of the linear effective action of the \(m\)-dimensional torus on \(S^{2m}\).

The stronger GKM$_4$ condition implies that the GKM graph of \(O\) has three-dimensional faces.
The restriction on the two-dimensional faces of the graph imply that the combinatorial types of the three-dimensional faces is also very restricted.
They are all combinatorially equivalent to one of the following: \(I^3,\Sigma^3,\Delta^3,\Sigma^2\times I,\Delta^2\times I\).

Using these restrictions, we show that the combinatorial type of a neighborhood of a vertex in the GKM graph \(\Gamma\)  of a non-negatively curved GKM$_4$ manifold is the same as that of a neighborhood of a vertex in the vertex edge graph \(\tilde{\Gamma}\) of a finite product \(\prod_i\Delta^{n_i}\times \prod_i\Sigma^{m_i}\).
Extending this local result to all of \(\tilde{\Gamma}\) then yields the covering described in Theorem \ref{sec:introduction}.

Since the number of the vertices in the GKM graph is equal to the total Betti number of the orbifold we get the following gap phenomenon:

\begin{cor}
  Let \(O\) of dimension \(2n\) be as in the previous theorem, then the total Betti number \(b(O)=\sum_ib_i(O)\) is either smaller than or equal to \(2^{n-2}\cdot 3\) or equal to \(2^n\). The latter case appears if and only if the GKM graph of \(O\) is combinatorially equivalent to the vertex-edge graph of \(I^n\).
\end{cor}

Note that the upper bound on the total Betti number is sharp.
Therefore we are verifying a conjecture of Gromov \cite{Gromov} in the special case of non-negatively curved GKM$_4$-manifolds. He conjectured that for general non-negatively curved manifolds of dimension \(d\) the total Betti number is bounded from above by \(2^{d}\).

By the GKM Theorem, the GKM graph determines the rational cohomology of a GKM orbifold.
Therefore, if we can show that all GKM graphs appearing in the above theorem can be realised as GKM graphs of certain model GKM orbifolds, any non-negatively curved GKM orbifold will have rational cohomology isomorphic to the rational cohomology of one of the model orbifolds.

To construct the models, we have to show that GKM$_4$ graphs with underlying graph equal to the vertex-edge graph of \(\prod_i\Delta^{n_i}\times\prod_i\Sigma^{m_i}\) extend --- in the sense of Kuroki \cite{Kuroki} --- to GKM$_n$ graphs, i.e. to GKM graphs of torus orbifolds over \(\prod_i\Delta^{n_i}\times \prod_i\Sigma^{m_i}\).
This reasoning then leads to

\begin{theorem}
\label{sec:introduction-1}
  If \(O\) is a GKM$_4$ orbifold such that the GKM graph of \(O\) is the vertex-edge graph of a product \(\prod_i\Delta^{n_i}\times \prod_i\Sigma^{m_i}\), then the rational cohomology of \(O\) is isomorphic to the rational cohomology of a non-negatively curved torus orbifold.
\end{theorem}

To get the models in the general case, we show that the deck transformation group \(G\) of the covering \(\tilde\Gamma\rightarrow \Gamma\) from Theorem \ref{sec:introduction} acts on the model torus orbifold \(\tilde{O}\) associated to the extended GKM graph \(\tilde\Gamma\) as in Theorem~\ref{sec:introduction-1}.
The quotient \(\tilde{O}/G\) is a GKM$_4$ orbifold realising the GKM graph \(\Gamma\).
Hence we get

\begin{theorem}
\label{sec:introduction-3}
  Let \(O\) be a non-negatively curved GKM$_4$ orbifold. Then there is a non-negatively curved torus orbifold \(\tilde{O}\) and an isometric action of a finite group \(G\) on \(\tilde{O}\) such that
  \begin{equation*}
    H^*(O;\mathbb{Q})\cong H^*(\tilde{O}/G;\mathbb{Q}).
  \end{equation*}
Moreover, if \(O\) is a manifold then \(\tilde{O}\) is a simply connected manifold.
\end{theorem}

In the literature it is often assumed that a GKM manifold has an invariant almost complex structure. This assumption results in the fact that in this case the weights of the GKM graph have preferred signs. We consider this special case in Section \ref{sec:gkm-manifolds-with-2}.
We show that in the situation of Theorem \ref{sec:introduction} this implies that the covering of \(\Gamma\) is trivial, and that the covering graph is the vertex edge graph of \(\prod_i\Delta^{n_i}\).
Moreover, the torus manifold corresponding to the extended GKM$_n$ graph will also admit an invariant almost complex structure.
Torus manifolds over \(\prod_i\Delta^{n_i}\) admitting an invariant almost complex structure were classified in \cite{choi10:_quasit}.
They are all diffeomorphic to so-called generalised Bott manifolds.
These manifolds \(X\) are total spaces of iterated \(\mathbb{C} P^{n_i}\)-bundles
\begin{equation*}
  X=X_k\rightarrow X_{k-1}\rightarrow \dots \rightarrow X_1\rightarrow X_0=\{pt\},
\end{equation*}
where each \(X_i\) is the total space of the projectivisation of a Whitney sum of \(n_i+1\) complex line bundles over \(X_{i-1}\).
Their cohomology rings can be easily determined from the Chern classes of the involved line bundles.
Indeed, if \(P(E)\) is the projectivisation of a complex vector bundle \(E\) of dimension \(n\) over a base space \(B\), then we have
\begin{equation*}
  H^*(P(E);\mathbb{Z})\cong H^*(B)[x]/(f(x)),
\end{equation*}
where \(x\) has degree two and \(f(x)=\sum_{i=0}^nc_i(E)x^{n+1-i}\).
Here \(c_i(E)\) denotes the \(i\)-th Chern class of \(E\).
By iterating this formula one gets the cohomology rings of all generalised Bott manifolds.
By the above discussion we get

\begin{theorem}
\label{sec:introduction-4}
  Let \(M\) be a non-negatively curved GKM$_4$ manifold which admits an invariant almost complex structure.
  Then the rational cohomology ring of \(M\) is isomorphic to the rational cohomology ring of a generalised Bott manifold.
\end{theorem}

In \cite{MR3355120} a classification of non-negatively curved simply connected torus manifolds was given.
A crucial step in the proof was to show that the orbit space of such a torus manifold is combinatorially equivalent to a product \(\prod_i \Delta^{n_i}\times \prod_i\Sigma^{m_i}\).
The proof of this intermediate result was very long and highly technical.
With the methods of the paper at hand we can give a short conceptual proof of this result.

The Bott conjecture asks if any simply connected non-negatively curved manifold  is rationally elliptic.
Therefore it is natural to consider the question if the above theorems also hold for rationally elliptic GKM$_4$ manifolds or orbifolds.

By \cite{GGKRW}, the two-dimensional faces of the corresponding GKM graphs have at most four vertices.
Moreover, since our arguments are purely graph-theoretic we conclude that all the above theorems also hold for rationally elliptic GKM$_4$ orbifolds instead of non-negatively curved ones.

The remaining sections of this paper are structured as follows.
In Section~\ref{Preliminaries} we gather background material about GKM manifolds and orbifolds as well as on torus manifolds and orbifolds.
Then in Section~\ref{sec:coverings-gkm-graphs} we prove Theorem \ref{sec:introduction}.

In Section~\ref{sec:ext} we show that a GKM$_4$ graph with underlying graph the vertex edge graph of a product \(\prod_{i}\Delta^{n_i}\times \prod_i\Sigma^{m_i}\) always extends to a GKM\(_n\) graph with \(n=\sum_i n_i+\sum_i m_i\).
This is then used in Section~\ref{sec:model} to prove Theorems \ref{sec:introduction-1} and \ref{sec:introduction-3}.

In Section~\ref{sec:upper-bound-dimens} we give an example of a non-negatively curved GKM$_4$ manifold which does not have the same rational cohomology as a torus manifold.
This shows that the group \(G\) from Theorem~\ref{sec:introduction-3} cannot be omitted.

In Section~\ref{sec:gkm-manifolds-with-2} we consider GKM manifolds with invariant almost complex structure and prove Theorem~\ref{sec:introduction-4}.
Moreover, in the last Section~\ref{sec:torus-manif-revis} we give a short proof of the ``big lemma'' which is used in the classification of non-negatively curved torus manifolds.
\nopagebreak

Throughout, cohomology will be taken with rational coefficients.

We would like to thank the anonymous referee for comments which helped to improve the presentation of this paper.

\section{Preliminaries}
\label{Preliminaries}

\subsection{GKM manifolds}
We begin with a review of GKM theory for manifolds; below we will describe the changes that are necessary to treat orbifolds as well.

Consider an effective action of a compact torus $T$ on a smooth, compact, orientable manifold $M$ of dimension $2n$ with finite fixed point set, such that the one-skeleton
\[
M_1= \{p\in M\mid \dim Tp\leq 1\}
\]
of the action is a union of invariant $2$-spheres. Given that the fixed point set is finite, the second condition is equivalent to the condition that for every fixed point, the weights of the isotropy representation are pairwise linearly independent. To such an action one associates its \emph{GKM graph}: its vertices are given by the fixed points of the action; to any invariant $2$-sphere -- which contains exactly two fixed points -- one associates an edge connecting the corresponding vertices. Finally, any edge is labeled with the weight of the isotropy representation in any of the two fixed points which corresponds to the two-dimensional submodule given by the tangent space of this sphere. These labels are linear forms on the Lie algebra $\mathfrak t$ of $T$ and are well-defined up to sign.

We need to abstract from group actions and consider the occurring graphs detached from any geometric situation, as in \cite{GuilleminZara} or \cite{BGH}.

For a graph $\Gamma$ we denote its set of vertices by $V(\Gamma)$ and its set of edges by $E(\Gamma)$. We consider only graphs with finite vertex and edge set, but we allow multiple edges between vertices. Edges are oriented; for $e\in E(\Gamma)$ we denote by $i(e)$ its initial vertex and by $t(e)$ its terminal vertex. The edge $e$, with opposite orientation, is denoted $\bar e$. For a vertex $v\in V(\Gamma)$ we denote the set of edges starting at $v$ by $E_v$.

\begin{definition} Let $\Gamma$ be a graph. Then a \emph{connection} on $\Gamma$ is a collection of bijective maps $\nabla_e:E_{i(e)}\to E_{t(e)}$, for $e\in E(\Gamma)$, such that 
\begin{enumerate}
\item \(\nabla_e(e)=\bar{e}\) and
\item \(\nabla_{\bar{e}}=\nabla_e^{-1}\), for all $e\in E(\Gamma)$.
\end{enumerate}
\end{definition}

\begin{definition}\label{defn:gkmgraph}
  Let $k\geq 2$. A \emph{GKM$_k$ graph} (\emph{GKM graph} for $k=2$) \((\Gamma,\alpha)\) consists of an \(n\)-valent connected graph \(\Gamma\) and a map \(\alpha: E(\Gamma) \rightarrow H^2(BT^m)/\{\pm 1\}\) such that
  \begin{enumerate}
  \item\label{item:1} If \(e_1,\ldots,e_k\) are edges of \(\Gamma\) which meet in a vertex \(v\) of \(\Gamma\) then the \(\hat{\alpha}(e_i)\), $i=1,\ldots,k$, are linearly independent. Here \(\hat{\alpha}(e_i)\in H^2(BT^m)\) denotes a lift of \(\alpha(e_i)\).  Note that this property is independent of the choice of \(\hat{\alpha}(e_i)\).
  \item There exists a connection $\nabla$ on $\Gamma$ such that for any two edges $e_1,e_2$ which meet in a vertex $v$ there are \(p,q\in \mathbb{Q}\) such that
  \begin{equation}
\label{eq:1}
\hat{\alpha}(\nabla_{e_1}(e_2))= p \hat{\alpha}(e_2) + q
\hat{\alpha}(e_1).
      \end{equation}
Note here that \(p,q\) are determined up to sign by \(\alpha\). Moreover, if we fix a sign for \(\hat{\alpha}(\nabla_{e_1}(e_2))\), \(\hat{\alpha}(e_2)\), and \(\hat{\alpha}(e_1)\), then \(p,q\) are uniquely determined.
\item For each edge \(e\) we have \(\alpha(e)=\alpha(\bar{e})\).
  \end{enumerate}
\end{definition}


Definition \ref{defn:gkmgraph} is slightly more general than usual, as $p$ and $q$ are allowed to be rational numbers. The reason will become clear below, when we consider orbifolds. The construction of the graph described above always results in GKM graphs:

\begin{proposition}
For any action of a compact torus $T$ on a smooth, compact, orientable manifold $M$ of dimension $2n$ with finite fixed point set, and whose one-skeleton $M_1$ is the union of invariant $2$-spheres, the graph associated to the action in the way prescribed above admits a connection for which Equation \eqref{eq:1} holds, with $p=\pm 1$ and $q$ an integer.
In particular the graph is a GKM graph.

Moreover, if the weights at every vertex are \(3\)-independent, i.e. if any three weights at every vertex are linearly independent, then the connection is unique.
\end{proposition}
\begin{proof}
  Let \(N\) be one of the invariant \(2\)-spheres and \(T_N\) the principal isotropy group of the \(T\)-action on \(N\).
  Moreover, let \(x_1,x_2\) be the two \(T\)-fixed points in \(N\).
  Then we have two \(T\)-representations \(T_{x_1}M\otimes\C\) and \(T_{x_2}M\otimes \C\) on the complexified tangent spaces at these fixed points.
  Let
  \begin{align*}
    T_{x_i}M\otimes \C&=\bigoplus_j W_{ij}&W_{ij}&=\{v\in T_{x_{i}}M\otimes \C;\;\; tv=\alpha_{ij}(t)v\text{ for all } t\in T\},
  \end{align*}
  be the decomposition of \(T_{x_i}M\otimes \C\) into weight spaces. Here the \(\alpha_{ij}\) are some homomorphisms \(T\rightarrow S^1\). By the GKM condition each \(W_{ij}\) has complex dimension one.
  Moreover, the derivatives of the \(\alpha_{ij}\) agree --up to sign-- with the weights of the edges starting in \(x_i\).

  Since the \(T_N\)-action on \(N\) is trivial, it follows from the proof of Proposition 2.2 of \cite{MR0234452} that
  \(TM\otimes \C|_N\)  splits \(T_N\)-equivariantly  as
  \[TM\otimes \C|_N\cong \bigoplus_{i=1}^k V_i\otimes E_i,\]
  where the \(V_i\) are complex irreducible \(T_N\)-representations and the \(E_i\) are complex vector bundles with trivial \(T_N\)-action.
  Therefore the isomorphism type of the \(T_N\)-representation on \(T_xM\otimes \C\) is independent of \(x\in N\).
  In particular, there is an isomorphism of \(T_N\)-representations
  \begin{equation*}
    T_{x_1}M\otimes \C\cong T_{x_2}M\otimes \C.
  \end{equation*}
  So the homomorphisms \(\alpha_{1j}\) and \(\alpha_{2j}\) agree after restriction to \(T_N\) (and reordering).

  Because \(T_N\) has codimension one in \(T\) there is a homomorphism \(\alpha:T\rightarrow S^1\) --unique up to complex conjugation on \(S^1\subset \C\)-- such that \(\ker \alpha=T_N\).
  By the definition of \(T_N\), this \(\alpha\) --or its conjugate \(\bar{\alpha}\)-- induces the \(T\)-action on \(T_{x_i}N\cong \C\), \(i=1,2\).
  Moreover, every homomorphism \(T\rightarrow S^1\) which is trivial on \(T_N\) factors through \(\alpha\).

  We now apply this to \(\alpha_{1j}\cdot\bar{\alpha}_{2j}\), where \(\cdot\) and \(\bar{~}\) denote multiplication and complex conjugation in \(S^1\), respectively.
  So we get some factorisation
  \begin{equation}
    \label{eq:5}
    \alpha_{1j}\cdot\bar{\alpha}_{2j}=\alpha^{k_j}
  \end{equation}
  Forming the derivative of this expression leads to equation (\ref{eq:1}) with \(p=\pm 1\) and \(q=\pm k_j\in \mathbb{Z}\).
    Moreover, in case the weights at \(x_2\) are \(3\)-independent, for given \(\alpha_{1j}\) there is at most one \(\alpha_{2j}\) for which equation (\ref{eq:5}) holds.
  So our claim follows.
\end{proof}

\begin{remark}
An alternative proof of this proposition was given in \cite[p.\ 6]{GuilleminZara}; there, the fact that \(q\in \mathbb{Z}\) followed from the Atiyah-Bott-Berline-Vergne localization theorem.
\end{remark}

\begin{remark}
  \label{sec:gkm-manifolds}
For GKM$_3$-graphs the connection $\nabla$ in Condition $2$ of \ref{defn:gkmgraph} is unique.
\end{remark}

\begin{remark}
A $T$-invariant almost complex structure on a manifold $M$ allows to speak about weights of the isotropy representation that are well-defined elements of $\mathfrak t^*$, not only up to sign. On the level of graphs we say that a GKM graph admits an invariant almost complex structure if there is a lift \(\hat{\alpha}:E(\Gamma)\rightarrow H^2(BT^m)\) of \(\alpha\) such that (\ref{eq:1}) holds with \(p=1\) and $q$ an integer and \(\hat{\alpha}(e)=-\hat{\alpha}(\bar{e})\), for all edges \(e\).
\end{remark}

The relevance of this type of actions is grounded in the fact that for manifolds $M$ with vanishing odd-dimensional (rational) cohomology, the (equivariant) cohomology is determined by the associated GKM graph. We define
\begin{definition}\label{defn:gkm}
We say that an action of a compact torus $T$ on a smooth, compact, orientable  manifold $M$ is \emph{of type GKM$_k$} (simply \emph{GKM} for $k=2$) if $H^{odd}(M)=0$, the fixed point set is finite, and for every fixed point any $k$ weights of the isotropy representation are linearly independent. 
\end{definition}
\begin{theorem}[{\cite[Theorem 7.2]{GKM}}]\label{thm:gkm} For a $T$-action of GKM type with fixed points $p_1,\ldots,p_r$, the inclusion $M^T\to M$ induces an injection 
\[
H^*_T(M)\to H^*_T(M^T)=\bigoplus_{i=1}^r H^*(BT)
\]
whose image is given by the set of tuples $(f_i)\in \bigoplus_{i=1}^rH^*(BT)$ such that $f_i - f_j=\alpha\beta$ with some \(\beta\in H^*(BT)\) whenever $p_i$ and $p_j$ are joined by an edge with label $\alpha$.
\end{theorem}
Obviously, the image in particular depends only on the labelled graph $\Gamma$, so it is sensible to use the notation $H^*_T(\Gamma)$ for the $H^*(BT)$-subalgebra of $\bigoplus_{v\in V(\Gamma)} H^*(BT)$ defined in the theorem above.

It is well-known that the vanishing of the odd rational cohomology implies that the canonical map $H^*_T(M)\to H^*(M)$ is surjective. In particular, the rational cohomology ring of $M$ is determined by the GKM graph of the action.

\subsection{GKM orbifolds}
The fact that GKM-theory works equally well for torus actions on orbifolds was already remarked in \cite[Section 1.2]{GuilleminZara}. 
In this paper we consider general (orientable) orbifolds $O$, which are given by orbifold atlases on a topological space, consisting of good local charts $(\tilde U_p,\Gamma_p)$ for any point $p\in O$. For the precise definition, and all basics on orbifolds and Lie group actions we use Section 2 of \cite{GGKRW} as a general reference.
Other sources of background material on orbifolds are \cite{MR3244330}, \cite{MR1062771}, \cite{MR1116630}, \cite{MR1401525}, \cite{MR3251232}. But note that the notation differs slightly from source to source.

We consider an action of a torus on a compact, orientable orbifold $O$, in the sense of \cite[Definition 2.10]{GGKRW}. In \cite{GGKRW} it is shown that orbits, as well as components of fixed point sets $O^H$, where $H\subset T$ is a connected Lie subgroup, are strong suborbifolds of $O$. Moreover, for any $p\in O$, with good local chart $\pi:(\tilde U_p,\Gamma_p)\to U_p$ there is an extension $\tilde T_p$ of $T_p$ by $\Gamma_p$, acting on $\tilde U_p$. The $\tilde T_p$-action fixes the single point $\tilde p$ in the preimage $\pi^{-1}(p)$. We thus obtain a well-defined isotropy representation of $\tilde T_p$ on $T_p\tilde U_p$. Its restriction to the identity component $\tilde T_p^o$ of $\tilde T_p$ has well-defined weights \(\alpha\), which we consider, via the projection $\tilde T_p^o\to T_p$, as elements in $\mathfrak t_p^*/\{\pm 1\}$.

With this definition of weights, Definition \ref{defn:gkm} applies to torus actions on orbifolds equally well, and we can speak about torus actions on orbifolds of type GKM$_k$.
This leads to the following definition.

\begin{definition}\label{defn:gkm_orb}
An action of a compact torus $T$ on a smooth, compact, orientable  orbifold $O$ is \emph{of type GKM$_k$} (simply \emph{GKM} for $k=2$) if $H^{odd}(O)=0$, the fixed point set is finite, and for every fixed point any $k$ weights of the isotropy representation are linearly independent. 
\end{definition}

To any such action we can associate an $n$-valent graph $\Gamma$ as in the case of manifolds, because any non-trivial torus action on a two-dimensional compact, orientable orbifold with a fixed point has exactly two fixed points, see \cite[Lemma 3.9]{GGKRW}. For the labelling, we rescale the weights as follows: for a weight $\alpha$ at a fixed point $p$, the intersection of \(\R \alpha\) with the integer lattice in \({\mathfrak t}^*\) is isomorphic to \(\mathbb{Z}\).
We let \(\alpha'\) be a generator of this group, and \(k\) be the number of components of the principal isotropy group of the \(T\)-action on the \(2\)-sphere to which the weight space of \(\alpha\) is tangent.
We define \(\beta=k\alpha'\).

 We remark that the factor $k$ is irrelevant for what follows: we include it in order for the GKM graph to encode the full isotropy groups. Considering $\beta$ as a homomorphism $T\to S^1$, its kernel is precisely the principal isotropy group of the corresponding $2$-sphere.

In order to construct a connection on $\Gamma$, we now restrict to actions of type GKM$_3$. Let $v$ be a vertex of the graph, corresponding to the fixed point $p\in O$, and $e$ an edge with $i(e)=p$, with label $\alpha\in \mathfrak t^*/\{\pm 1\}$. Let $q$ be the fixed point corresponding to $t(e)$. For any other edge $e'$ at $v$, with weight $\beta$, we consider the connected subgroup $H\subset T$ with Lie algebra $\ker \alpha \cap \ker \beta$. By the GKM$_3$-condition and the slice theorem for actions on orbifolds \cite[Theorem 2.18]{GGKRW} the connected component of $O^H$ is a four-dimensional strong suborbifold of $O$. It contains $q$, and there exists an edge $f$ with $i(f) = t(e)$, with weight $\gamma$, such that $\ker \gamma \cap \ker \alpha = \mathfrak h$. We define a connection on $\Gamma$ by $\nabla_e e':=f$. By construction, $\gamma$ is a (rational) linear combination of $\alpha$ and $\beta$, so that Equation \eqref{eq:1} holds.

Also Theorem \ref{thm:gkm} holds true for GKM actions on orbifolds. This was (for torus orbifolds) already observed in \cite[Theorem 4.2]{GGKRW}.

\subsection{Torus manifolds and orbifolds}
\label{sec:torus-manif-orbif}

Here we gather the facts we need to know about torus manifolds and torus orbifolds.
General references for the constructions used here are \cite{BuchstaberPanov1}, \cite{BuchstaberPanov2}, \cite{MasudaPanov}, \cite{DavisJanuszkiewicz}, \cite{GGKRW}.

We start with a general construction of such manifolds and orbifolds.
An \(n\)-dimensional manifold with corners \(P\) is called nice if at each vertex of \(P\) there meet exactly \(n\) facets of \(P\), that is, exactly \(n\) codimension-one faces.
The faces of \(P\) ordered by inclusion form a poset \(\mathcal{P}(P)\), the so-called face poset of \(P\). We also denote by \(\mathcal{F}=\mathcal{F}(P)\) the set of facets of \(P\) and let \(m=|\mathcal{F}|\).

Now let \(P\) be a nice manifold with corners with only contractible faces, and assume that there is a map \(\lambda:\mathcal{F}\rightarrow \mathbb{Z}^n\) such that for every vertex \(v\) of \(P\),
\[\lambda(F_1),\dots,\lambda(F_n)\]
are linearly independent, where the \(F_i\) are the facets of \(P\) which meet in \(v\).

Then we can construct a torus orbifold, i.e. an orientable \(2n\)-dimensional orbifold \(O\) with an action of the \(n\)-dimensional torus \(T^n=\mathbb{R}^n/\mathbb{Z}^n\) with \(O^{T^n}\neq \emptyset\), such that the orbit space of the \(T^n\) action on \(O\) is homeomorphic to \(P\).
The orbifold \(O\) is defined as
\begin{equation*}
  O=(P\times \mathbb{R}^n/\mathbb{Z}^n)/\sim,
\end{equation*}
where \((x,v)\sim (x',v')\) if and only if \(x=x'\) and \(v-v'\) is contained in the \(\mathbb{R}\)-span of all the \(\lambda(F)\) with \(x\in F\).
Here \(T^n\) acts on the second factor of \(O\). Note that replacing the $\lambda(F)$ by nonzero multiples does not change the associated orbifold $O$.

If for every vertex \(v\) the \(\lambda(F_i)\) of the facets \(F_i\) which meet at \(v\) form a basis of \(\mathbb{Z}^n\), then \(O\) is a manifold and the \(T^n\)-action is locally standard, i.e., locally modelled on effective \(n\)-dimensional complex representations of \(T^n\).

The preimages of the facets of \(P\) under the orbit map are invariant suborbifolds of codimension two in \(O\).
Therefore their equivariant Poincar\'e duals \(v_1,\dots,v_m\in H^2_T(O;\mathbb{Q})\) are defined.
Moreover, these Poincar\'e duals form a basis of \(H^2_T(O;\mathbb{Q})\) because the faces of \(P\) are contractible (see \cite{MasudaPanov} and \cite{MR2791564}).

The contractibility of the faces of \(P\) also implies that the rational cohomology of a torus orbifold \(O\) as above is concentrated in even degrees.

Equivalently to giving the labels \(\lambda\) of the facets, one can also define \(O\) by a labeling \(\hat\alpha\)  of the edges of \(O\) in such a way that for edges \(e_1,\dots,e_n\) meeting a vertex, \(\hat\alpha(e_1),\dots,\hat\alpha(e_n)\) is the basis of \({\mathbb{Q}^n}^*\) dual to \(\lambda(F_1),\dots,\lambda(F_n)\in \mathbb{Z}^n\otimes \mathbb{Q}\).
In this way the vertex-edge graph of \(P\) becomes a so-called torus graph
 (see \cite{MaedaMasudaPanov}), i.e. a GKM graph of a torus manifold or orbifold.
 
Similar to the definition of the torus orbifold \(O\) associated to the pair \((P,\lambda)\) one can associate a moment angle manifold of dimension \(n+m\) to \(P\). 
This goes as follows.

Denote the facets of \(P\) by \(F_1,\dots,F_m\) and for each \(i=1,\dots,m\) let \(S^1_i\) be a copy of the circle group.
Then define
\begin{equation*}
  Z_P=(P\times (\prod_{i=1}^{m} S^1_i))/ \sim,
\end{equation*}
where \((x,t)\sim (x',t')\) if and only if \(x=x'\) and \[t't^{-1}\in \prod_{i;\;x\in F_i}S^1_i\subset \prod_{i=1}^m S^1_i.\]
There is an action of \(T^m=\prod_{i=1}^m S^1_i\) on \(Z_P\) induced by multiplication on the second factor.
Moreover the torus orbifold \(O\) from above is the quotient of the action of the kernel of a homomorphism \(\varphi:T^m\rightarrow T^n=\mathbb{R}^n/\mathbb{Z}^n\). Here \(\varphi\) is defined by the condition that its restriction to \(S^1_i\) induces an isomorphism \(S^1_i\rightarrow \R\lambda(F_i)/(\mathbb{Z}^n\cap \R\lambda(F_i))\).

Note that if \(P\) is a product of simplices \(\Delta^{n_i}\) and quotients \(\Sigma^{m_i}=S^{2m_i}/T^{m_i}\), then \(Z_P\) is a product of spheres. We can equip this product with the product metric of the round spheres.
If we do so, \(T^m\) is identified with a maximal torus of the isometry group of \(Z_P\).

\begin{example}
  At the end of this section we give examples of torus orbifolds.
  Let \(P=\Delta^2\) be a triangle. Then we have \(n=2\) and \(m=3\). Denote by \(F_1,F_2,F_3\) the facets of \(P\). Moreover, let
  \begin{align*}
    \lambda(F_1)&=(1,0),&\lambda(F_2)&=(0,1),&\lambda(F_3)&=(\alpha,\beta)
  \end{align*}
  with \(\alpha,\beta \in \mathbb{Z}-\{0\}\).
  Then by the above discussion the pair \((P,\lambda)\) defines a torus orbifold \(O\). Note that \(O\) is a torus manifold if and only if \(|\alpha|=|\beta|=1\).

  The moment angle manifold \(Z_P\) associated to \(P\) is \(S^5\subset \C^3\) with a linear \(T^3=\mathbb{R}^3/\mathbb{Z}^3\)-action.
  The map \(Z_P\rightarrow O\) constructed above is the orbit map for the action of the subtorus of \(T^3\) whose Lie-algebra is generated by \((-\alpha,-\beta,1)\).
  Therefore \(O\) is a so-called weighted projective space.
\end{example}

\section{Coverings of GKM graphs}
\label{sec:coverings-gkm-graphs}


In this section we construct a covering of a GKM$_k$ graph, $k\geq 4$, with small three-dimensional faces by the vertex edge graph of a product \(\prod_i \Delta^{n_i}\times \prod_i\Sigma^{m_i}\).
We start with the definition of what we mean by the faces of a graph.

\begin{definition}
Let $\Gamma$ be a graph with a connection $\nabla$. Then an \emph{$l$-dimensional face} of $\Gamma$ is a connected $l$-valent $\nabla$-invariant subgraph of $\Gamma$.
\end{definition}

\begin{lemma}\label{lem:uniquefaceofabstractgkmgraph}
Let $\Gamma$ be a GKM$_k$-graph, where $k\geq 3$. Then for any vertex $v$ of $\Gamma$ and any edges $e_1,\ldots e_l$ that meet at $v$, where $1\leq l\leq k-1$, there exists a unique $l$-dimensional face of $\Gamma$ that contains $e_1,\ldots,e_l$.
\end{lemma}
\begin{proof}
  First recall that by Remark~\ref{sec:gkm-manifolds} there is a unique connection on \(\Gamma\) which is compatible with the weights.
  Therefore it makes sense to speak about the faces of a GKM$_k$-graph, \(k\geq 3\).
  
  Let $V$ be the ($l$-dimensional) span of the $\alpha(e_i)$ in $H^2(BT^m)$. Consider the subgraph of $\Gamma$ that consists of those edges whose labeling is contained in $V$, and let $\tilde\Gamma$ be its connected component of $v$. We claim that this subgraph is $\nabla$-invariant and $l$-valent.

As $\Gamma$ is GKM$_k$, with $k>l$, the only edges at $v$ contained in $\tilde\Gamma$ are $e_1,\ldots,e_l$. Moreover, whenever $w$ is an $l$-valent vertex of $\tilde\Gamma$ and $e\in E(\Gamma)$ with $i(e) = w$, then also $t(e)$ is $l$-valent. In fact, \eqref{eq:1} shows that for any edge $e'$ at $w$ in $E(\tilde\Gamma)$, also $\nabla_e e'$ is an edge of $E(\tilde\Gamma)$, and the GKM$_k$ property of $\Gamma$ shows that there is no further edge at $t(e)$ contained in $\tilde\Gamma$.  
\end{proof}

A simply connected space \(X\) is called rationally elliptic if it has finite dimensional rational homotopy and finite dimensional rational cohomology, i.e.,
\begin{align}
  \label{eq:3}
  \sum_{i\geq 2}\dim \pi_i(X)\otimes \mathbb{Q}&<\infty &\text{and}&&\sum_{i\geq 2}\dim H^i(X)\otimes \mathbb{Q}&<\infty.
\end{align}
In this case the minimal Sullivan model of \(X\) is elliptic.
Moreover, for simply connected spaces the other implication also holds.
However, there are non-simply connected spaces whose minimal Sullivan model is elliptic which do not satisfy the conditions (\ref{eq:3}).

\begin{lemma}
\label{sec:gkm-graphs}
  Let \(O\) be a GKM$_3$ manifold or orbifold which admits an invariant metric of non-negative curvature or whose minimal Sullivan model is elliptic. Then each two-dimensional face of the GKM graph of \(O\) has at most four vertices. 
\end{lemma}
\begin{proof}
  The proof is essentially the same as the proof of Lemma 4.2 in \cite{MR3355120}.
  In the non-negatively curved case it was first discussed in \cite{MR2784821}.
  It has been translated to the orbifold setting in \cite{GGKRW}.
  Here we repeat it for the sake of completeness.
  
  The two-dimensional faces of the GKM graph of \(O\) are GKM graphs of four-dimensional invariant totally geodesic suborbifolds of \(O\).
  These submanifolds are fixed point components of codimension-two subtori of \(T\).
  They are non-negatively curved if \(O\) is non-negatively curved.
  Let \(O'\) be one of these suborbifolds.
  Then \(O'/T\) is a two-dimensional non-negatively curved Alexandrov space with totally geodesic boundary, such that the points in \((O')^T\) correspond to corners of the orbit space, i.e. points  on the boundary whose space of directions has diameter \(\pi/2\).
  Let \(a_0,\dots,a_{k-1}\) be these corners such that for \(i\in\mathbb{Z}/k\mathbb{Z}\), \(a_i\) is connected to \(a_{i+1}\) by a totally geodesic arc which is contained in the boundary.
  For \(i\neq 0\) choose geodesics \(\gamma_i\) from \(a_0\) to \(a_i\), such that \(\gamma_1\) and \(\gamma_{k-1}\) are part of the boundary.

  Then by Toponogov's Theorem the sum of angles in each of the triangles spanned by \(\gamma_i,\gamma_{i+1}\) and the part of the boundary between \(a_i\) and \(a_{i+1}\) is at least \(\pi\). Summing over all these triangles we get the inequality
  \begin{equation*}
    \pi(k-2)\leq \frac{\pi}{2} k.
  \end{equation*}
Hence the claim follows in this case.

Now assume that the minimal Sullivan model of \(O\) is elliptic. Then, by \cite{Allday}, the minimal model
\((\Lambda V,d)\) of \(O'\) is also elliptic. Since \(\chi(O'^T) = \chi(O' )\), the number of vertices in the GKM graph of \(O'\) is equal to the Euler characteristic of \(O'\). Moreover, since \(H^*(O)\) is evenly graded, \(H^{\text{odd}}(O')=0\) by localization in equivariant cohomology. In particular, \(H^1(O')=0\).

Hence, by \cite[Theorem
32.6 (ii)]{FHT} and \cite[Proposition 12.2]{FHT}, we have
\begin{equation*}
4 \geq 2\dim V^2=2b_2(O').  
\end{equation*}

Therefore we have \(\chi(O')\leq 4\) and there are at most four vertices in the GKM graph of \(O'\).
\end{proof}

\begin{lemma}
\label{sec:gkm-graphs-1}
Let \(O\) be a orientable GKM$_4$ orbifold such that all two-dimensional faces of the GKM graph of \(O\) have at most four vertices. Then each three-dimensional face of the GKM graph of \(M\) has one of the following combinatorial types: \(\Delta^3\), \(\Sigma^3\), \(\Delta^2\times I\), \(\Sigma^2\times I\), \(I^3\).
Moreover, the face structure of the three-dimensional faces induced by the connection is the natural one.
\end{lemma}
\begin{proof}
  First note that the three-dimensional faces of the GKM graph of \(O\) are GKM graphs of six-dimensional torus orbifolds \(N_1,\dots,N_k\).

  Note that since \(O\) is orientable, all \(N_i\) (and therefore all the orbit spaces \(N_i/T\)) are orientable orbifolds without (and with, respectively) boundary.
  Moreover, because the cohomology of \(O\) (and therefore that of the \(N_i\)) is concentrated in even degrees, the orbit space \(N_i/T\) and all its faces are acyclic over the rationals.
  Because two-dimensional orbifolds are homeomorphic to surfaces it follows from the classification of surfaces that the facets, i.e. 2-dimensional faces, of \(N_i/T\) are homeomorphic to discs \(D^2\).
  Moreover, by Lemma~\ref{lem:uniquefaceofabstractgkmgraph}, any two edges of the GKM graph of \(N_i\) meeting in a vertex span a unique \(2\)-dimensional face of \(N_i/T\).

  There are the following cases:
  \begin{enumerate}
  \item There is a facet \(F_1\) of \(N_i/T\), which has two vertices. And, there is another facet \(F_2\) of \(N_i/T\), which intersects with \(F_1\) in an edge and has
    \begin{enumerate}
    \item two vertices, or
    \item three vertices, or
    \item four vertices.
    \end{enumerate}
  \item  There is a facet \(F_1\) of \(N_i/T\), which has three vertices. And, there is another facet \(F_2\) of \(N_i/T\), which intersects with \(F_1\) in an edge and has
    \begin{enumerate}
    \item three vertices, or
    \item four vertices.
    \end{enumerate}
  \item All facets of \(N_i/T\) have four vertices.
  \end{enumerate}

In case 1.a), it is clear that \(P=N_i/T\) is combinatorially equivalent to \(\Sigma^3\).
Moreover, in case 1.c), using the $3$-valence of the graph of \(N_i\) one easily sees that \(P\) is of type \(\Sigma^2\times I\).

Next assume that \(F_2\) is of type \(\Delta^2\), that is, we are in case 1.b). Let \(v_0\) be the vertex of \(F_2\) which does not belong to \(F_1\) and \(v_1\) and \(v_2\) be the other vertices. Then one sees by considering the faces which meet at \(v_1\) that the two edges \(v_0v_1\) and \(v_0v_2\) belong to two different faces which both contain all three vertices, as in the first graph in Figure \ref{false3dface}.
But any two edges at \(v_0\) must span a unique face, a contradiction.

Alternatively, the following argument also leads to a contradiction.
By the \(3\)-valence of the graph there is a third edge \(e\)  starting from \(v_0\).
Since it is contained in a two-dimensional face of \(N_i/T\), the end point of \(e\) must be \(v_1\) or \(v_2\) (or must be connected via one edge to one of \(v_1,v_2\)).
But this cannot happen because of the \(3\)-valence of the graph (at \(v_1\) and \(v_2\)).

Hence the case 1.b) does not occur.

 \begin{figure}[htb]
 \begin{center}
 \includegraphics[height=100pt]{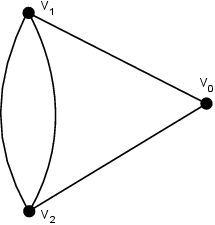}
 \qquad\qquad\qquad
 \includegraphics[height=100pt]{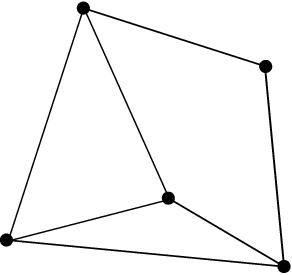}
\end{center} 
  \caption{These graphs do not occur}
 \label{false3dface}
 \end{figure}

 We now assume that \(F_1\) is of type \(\Delta^2\).
 That is, we are in one of the cases 2.a) or 2.b).
 
 If all other facets of \(P\) which have an edge with \(F_1\) in common are also of type \(\Delta^2\), then \(P\) must be of type \(\Delta^3\).
If all faces which have an edge with \(F_1\) in common are of type \(I^2\),  then \(P\) is of type \(\Delta^2\times I\).

Therefore we have to exclude the case that there is a face \(F_2\) of type \(\Delta^2\) and a face \(F_3\)  of type \(I^2\) such that \(F_1\cap F_2\cap F_3\) is a vertex, as in the second graph in Figure \ref{false3dface}.

In this case the third face \(F_4\) which has an edge with \(F_1\) in common must be of type \(I^2\).
Using the \(3\)-valence of the graph one now gets a contradiction in a similar way as in case 1.b).

In the remaining case 3) let \(F_1\) be one of the two-dimensional faces of \(N_i/T\).
Then by Lemma~\ref{lem:uniquefaceofabstractgkmgraph} there are four not necessarily pairwise distinct two-dimensional faces \(F_2,\dots,F_5\)   of \(N_i/T\) which have non-trivial intersection with \(F_1\).
If all the intersections \(F_i\cap F_j\), \(i=1,\dots,5\), are connected or empty, then \(F_2,\dots,F_5\) are pairwise distinct and it is clear that \(P\) is of type \(I^3\).

Therefore assume that \(F_1\cap F_2\) has two components.
These two components must then both be edges of \(P\).

Let \(v_1,\dots,v_4\) be the vertices of \(F_1\) such that \(v_i\) and \(v_{i+1}\) are connected by an edge \(v_iv_{i+1}\) for all \(i\in \mathbb{Z}/4\mathbb{Z}\).

Assume that \(v_1v_2\) and \(v_3v_4\) are contained in the intersection of \(F_1\) and \(F_2\).
If \(v_1\) and \(v_4\) were connected by an edge in \(F_2\), then there would be a facet of type \(\Sigma^2\) in \(P\), contradicting our assumption.
Therefore \(v_1\) and \(v_3\) are connected by an edge in \(F_2\) (and similarly \(v_2\) and \(v_4\)).
So \(F_1\cup F_2\) is homeomorphic to a Moebius strip, contradicting our orientability assumption on \(O\).
\end{proof}

\begin{definition}
\label{sec:gkm-graphs-6}
Let $\Gamma$ be a connected graph with a connection $\nabla$. We say that $\Gamma$ is a \emph{graph with small three-dimensional faces} if the following conditions hold true:
\begin{enumerate}
\item For any $x\in V(\Gamma)$ and distinct edges $e_1,e_2,e_3$ meeting at $x$, there exits a unique $3$-dimensional face of $\Gamma$ containing $e_1,e_2$ and $e_3$.
\item The conclusion of Lemma \ref{sec:gkm-graphs-1} holds true, i.e., any three-dimensional face of $\Gamma$ has the combinatorial type of $\Delta^3, \Sigma^3$, $\Delta^2\times I$, $\Sigma^2\times I$ or $I^3$ (see Figure \ref{small3dfaces}) and these faces have the natural face structure.
\end{enumerate}
\end{definition}

\begin{figure}[htb]
 \begin{center}
 \includegraphics[height=100pt]{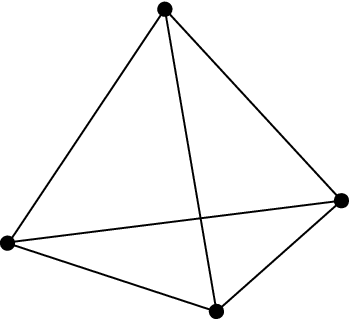}
 \qquad \qquad
 \includegraphics[height=100pt]{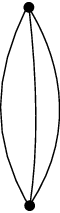}
 \qquad \qquad
 \includegraphics[height=100pt]{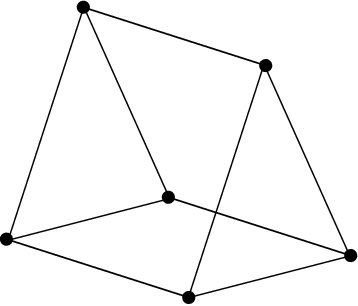}
 \qquad
 \qquad
 \includegraphics[height=100pt]{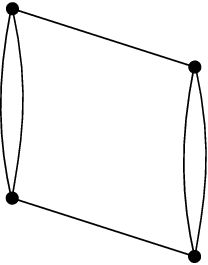}
 \qquad
 \qquad
 \includegraphics[height=100pt]{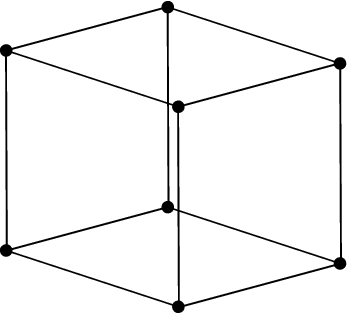}

\end{center} 
  \caption{Small three-dimensional faces}
 \label{small3dfaces}
\end{figure}

Lemmas \ref{lem:uniquefaceofabstractgkmgraph} and \ref{sec:gkm-graphs-1} thus say that the GKM graph of a connected nonnegatively curved GKM$_4$ orbifold is a graph with small three-dimensional faces.

\begin{definition}
  Let \(x\in V(\Gamma)\), where $\Gamma$ is a graph with small three-dimensional faces. We call a subgraph \(\Gamma'\subset \Gamma\) a \emph{local factor} at \(x\) if \(\Gamma'\) contains \(x\), is \(\nabla\)-invariant, has the combinatorial type of \(\Delta^k\) or \(\Sigma^k\), \(\nabla\) induces the natural face structure on \(\Gamma'\) and \(\Gamma'\) is maximal with these properties. 
\end{definition}

\begin{lemma}
\label{sec:gkm-graphs-3}
  Let $\Gamma$ be a graph with small three-dimensional faces. For two edges \(e,e'\)  emanating from \(x\) we have:
  \begin{enumerate}
  \item \(e\) and \(e'\) belong to the same local factor of type \(\Delta^k\) if and only if \(e\) and \(e'\) span a triangle.
  \item \(e\) and \(e'\) belong to the same local factor of type \(\Sigma^k\)  if and only if \(e\) and \(e'\) span a biangle.
  \item \(e\) and \(e'\) do not belong to the same local factor if and only \(e\) and \(e'\) span a square.
  \end{enumerate}
\end{lemma}
\begin{proof}
  By definition, all two dimensional faces of \(\Gamma\)  have at most four vertices.
  Therefore the lemma follows from the definition of the local factors, because two-dimensional faces of \(\Delta^k\) and \(\Sigma^k\) are triangles and biangles, respectively.
\end{proof}

\begin{lemma}
\label{sec:gkm-graphs-4}
Let $\Gamma$ be a graph with small three-dimensional faces. Then the local factors at \(x\in V(\Gamma)\) are partitioning \(E_x=G_1\amalg\dots\amalg G_{n_x}\) in such a way that each \(G_i\) contains the edges which span a local factor at \(x\).
\end{lemma}
\begin{proof}
  We have to show that belonging to the same local factor is an equivalence relation on \(E_x\).
  We only have to show transitivity.
  
  Let \(e,e',e''\in E_x\) be such that \(e\) and \(e'\) belong to one local factor \(\sigma_i\) and \(e'\) and \(e''\) belong to another local factor.
Then, by Lemma~\ref{sec:gkm-graphs-3}, \(e\) and \(e'\) and \(e'\) and \(e''\) span a triangle or a biangle, respectively.

Consider the three-dimensional face \(F\)  of \(\Gamma\) spanned by \(e,e',e''\). It has one of the combinatorial types described in Definition~\ref{sec:gkm-graphs-6}.
Since none of the faces spanned by $e,e'$ and $e',e''$ are squares, it follows that \(F\) has the combinatorial type of \(\Delta^3\) or \(\Sigma^3\).
Hence, we have shown transitivity and the claim follows.
\end{proof}

\begin{lemma}
\label{sec:gkm-graphs-2}
Let \(\Gamma\) be a graph with small three-dimensional faces and let \(e\) be an oriented edge of \(\Gamma\).
Then the connection \(\nabla_e:E_{i(e)}\rightarrow E_{t(e)}\) preserves the partitions \(E_{i(e)}=G_1\amalg\dots\amalg G_{n_{i(e)}}\) and \(E_{t(e)}=G'_1\amalg\dots\amalg G'_{n_{t(e)}}\).
Moreover, the combinatorial types of the local factors spanned by \(G_i\) and \(\nabla_e(G_i)\) are the same.
\end{lemma}
\begin{proof}
  Let \(e',e''\) be other edges of \(\Gamma\) emanating from \(i(e)\).
  By Lemma~\ref{sec:gkm-graphs-3}, we have to show that the two-dimensional faces spanned by any choice of two edges \(e_1,e_2\) from \(e,e',e''\), respectively, has the same combinatorial types as the two-dimensional face spanned by \(\nabla_e(e_1),\nabla_e(e_2)\).

  To do so, we consider the three-dimensional face of \(\Gamma\) spanned by \(e,e',e''\). It contains the two-dimensional faces spanned by \(e,e'\) and \(e,e''\), respectively.
  Moreover, these two-dimensional faces are also spanned by \(e,\nabla_e(e')\) and \(e,\nabla_e(e'')\), respectively.

The three-dimensional face also contains the face spanned by \(e'\) and \(e''\).
  By an inspection of the list of three-dimensional faces in Definition \ref{sec:gkm-graphs-6}, it follows that the two-dimensional faces spanned by \(e',e''\) and \(\nabla_e(e'),\nabla_e(e'')\), respectively, have the same combinatorial types.
  Hence the lemma is proved.
\end{proof}

\begin{lemma}
  \label{sec:gkm-graphs-their}
Let \(\Gamma\) be a graph with small three-dimensional faces. Let \(e, e'\) and \(f\) be edges in \(\Gamma\) with the same initial point.
\begin{enumerate}
\item If \(e\) and \(e'\) span a biangle, then \(\nabla_ef=\nabla_{e'}f\).
\item If \(e\) and \(e'\) span a square, then \(\nabla_{e_1'}\nabla_{e}f=\nabla_{e_1}\nabla_{e'}f\), where \(e_1,e_1'\) are the edges opposite to \(e\) and \(e'\), respectively, in the square spanned by \(e\) and \(e'\).
\end{enumerate}
\end{lemma}
\begin{proof}
To see the first claim, one has to show that if \(e,e'\) span a biangle \(\Sigma^2\) and \(f\) is an edge with the same initial point as \(e\) and \(e'\) then
\begin{equation*}
  \nabla_{e}f=\nabla_{e'}f.
\end{equation*}
To see this, first assume that \(e\) and \(f\) span a biangle. Then by Definition \ref{sec:gkm-graphs-6}, \(e\), \(e'\), \(f\) span a face of \(\Gamma\) which is combinatorially equivalent to \(\Sigma^3\). Hence, \(\nabla_{e}f=\bar{f}=\nabla_{e'}f\) follows.

Next assume that \(e\) and \(f\) span a square. Then, by Definition \ref{sec:gkm-graphs-6}, \(e\), \(e'\), and \(f\) span a face with the combinatorial type of \(\Sigma^2\times I\).
Therefore  \(\nabla_{e}f=g=\nabla_{e'}f\), where \(g\) is the third edge at \(t(e)=t(e')\).
Hence the claim follows in this case.

By Definition \ref{sec:gkm-graphs-6}, the case that \(e\) and \(f\) span a triangle does not occur. So the claim follows.

The second claim follows in a similar way, again by considering the three-dimensional faces of \(\Gamma\).  
\end{proof}

The following theorem states that any graph with small three-dimensional faces is covered by a product of \(\Delta^k\)s and \(\Sigma^k\)s. Here by a covering of a graph by another graph we mean the following: We consider graphs as one-dimensional CW-complexes, and coverings should be cellular. Note that the graphs we consider are $n$-valent, for some $n\geq 1$; for $n\neq 2$ a covering of $n$-valent graphs is automatically cellular, because in this case the neighbourhoods of points in the interior of a one-cell and the neighbourhoods of the vertices are not homeomorphic.

\begin{theorem}\label{thm:covering}
Let $\Gamma$ be a graph with small three-dimensional faces, $x\in V(\Gamma)$, and $\sigma_1,\ldots,\sigma_k$ the local factors at $x$. Consider the product graph $\tilde\Gamma:=\prod_{i=1}^k \sigma_i$, equipped with its natural connection $\tilde\nabla$. Let $x_0\in \tilde\Gamma$ be a base point, and $f:E(\tilde\Gamma)_{x_0}\to E(\Gamma)_x$ a bijection sending the edges of a local factor to the edges of a local factor. Then there exists a unique covering $\pi:\tilde\Gamma\to \Gamma$ extending $f$ that is compatible with the connections, i.e., which satisfies $\nabla_{\pi(e)} \circ \pi = \pi\circ \tilde\nabla_e$ for all edges $e\in E(\tilde\Gamma)$.
\end{theorem}
\begin{proof} The compatibility condition $\nabla_{\pi(e)} \circ \pi = \pi\circ \tilde\nabla_e$ shows that if $e,e'$ are edges meeting at some vertex $v$, and $\pi(e)$ and $\pi(e')$ are given, then $\pi(\nabla_e e')$ is uniquely determined. This implies the uniqueness of $\pi$.

  We have to show the existence of $\pi$. Note that for any path $\gamma$ in $\tilde\Gamma$, say from \(x_0\) to \(y\), the connection \(\tilde{\nabla}\) on \(\tilde\Gamma\) induces a bijection
  
\[\tilde{\nabla}_{\gamma}:E(\tilde\Gamma)_{x_0}\rightarrow E(\tilde\Gamma)_{y}\quad \tilde{\nabla}_{\gamma}= \tilde{\nabla}_{e_m}\circ \dots \circ  \tilde{\nabla}_{e_1},\]
where \(\gamma=e_1*\dots*e_m\).
Similarly, for paths in \(\Gamma\) from \(x\) to \(y\) there is a bijection
\[\nabla_{\gamma}:E(\Gamma)_{x}\rightarrow E(\Gamma)_{y}\]
induced by \(\nabla\).

For $n\geq 0$, we let $\tilde\Gamma_n$ be the subgraph of $\tilde\Gamma$ whose vertices are those that have distance at most $n$ to $x_0$, and whose edges are all the edges of $\Gamma$ connecting two such vertices. We prove by induction that we can construct a map of graphs $\pi:\tilde\Gamma_n\to \Gamma$ extending $f$ that satisfies
\[
\nabla_{\pi(\gamma)} \circ \pi (e)= \pi \circ \tilde\nabla_\gamma e
\]
for all shortest paths $\gamma$ in $\tilde\Gamma_n$ starting at $x_0$, and all edges $e$ of $\tilde\Gamma_n$ with $i(e) = i(\gamma)$. For $n=0$ there is nothing to do. For \(n=1\), the existence of \(\pi\) is guaranteed by the fact that \(f\) sends local factors to local factors.

We assume that $\pi:\tilde\Gamma_{n-1}\to \Gamma$ is already constructed, and wish to construct $\pi:\tilde\Gamma_{n}\to \Gamma$. Let $e$ be an edge of $\tilde\Gamma_n$ which is not an edge of $\tilde\Gamma_{n-1}$, but whose initial vertex $i(e)$ is a vertex of $\tilde\Gamma_{n-1}$, and choose a shortest path $\gamma$ from $x_0$ to $i(e)$. Note that $\gamma$ is a path in $\tilde\Gamma_{n-1}$. We want to define \(\pi(e)=\nabla_{\pi(\gamma)}\pi(\tilde{\nabla}_\gamma^{-1} e)\); in order to do so we have to show that this definition is independent of the choice of $\gamma$.

If \(x',y'\) are vertices of \(\tilde\Gamma\), then the shortest paths between \(x'\) and \(y'\) are of the following form:

\begin{equation*}
  e_{i_1}*\dots * e_{i_l},
\end{equation*}
where

\begin{itemize}
\item  \(e_{i_j}\) is an edge tangent to \(\sigma_{i_j}\)
\item the \(i_j\) are unique up to ordering and \(i_j\neq i_{j'}\) if \(j\neq j'\)
\item If \(\sigma_{i_j}=\Delta^k\), then \(e_{i_j}\) is unique
\item If \(\sigma_{i_j}=\Sigma^k\), \(k\geq 2\), then by Lemma \ref{sec:gkm-graphs-their}, \(\nabla_{e_{i_j}}\) does not depend on the choice of edge in \(\sigma_{i_j}\). 
\end{itemize}

By Lemma~\ref{sec:gkm-graphs-their}, we have
\begin{equation*}
  \tilde{\nabla}_{e_{i_j}}\circ \tilde{\nabla}_{e_{i_{j+1}}} = \tilde{\nabla}_{e_{i_{j+1}}'}\circ \tilde{\nabla}_{e_{i_j}'},
\end{equation*}
and
\begin{equation*}
  \nabla_{\pi(e_{i_j})}\circ \nabla_{\pi(e_{i_{j+1}})} = {\nabla}_{\pi(e_{i_{j+1}}')}\circ {\nabla}_{e_{i_j}'},
\end{equation*}
where \(e_{i_j}'\) and \(e_{i_{j+1}}'\) are the opposite edges to \(e_{i_j}\) and \(e_{i_{j+1}}\), respectively, in the square spanned by \(e_{i_j}\) and \(e_{i_{j+1}}\). Hence \(\tilde{\nabla}_\gamma\) and \(\nabla_{\pi(\gamma)}\) do not depend on the chosen shortest path from $x_0$ to $i(e)$.

Next, we show that for two such edges $e,e'$ with initial point in $\tilde\Gamma_{n-1}$ and same end point $t(e) = t(e')$, which is not a vertex of $\tilde\Gamma_{n-1}$, we have $\pi(t(e)) = \pi(t(e'))$. To do so let \(\gamma\) and \(\gamma'\) be two minimising curves from \(x_0\) to \(i(e)\) and \(i(e')\), respectively.
Then
\begin{align*}
  \gamma&*e&\gamma'&*e'
\end{align*}
are minimising curves from \(x_0\) to \(t(e)\).
Therefore they coincide up to ordering of the edges (replacing edges by parallel edges) and possible disambiguity with multiple edges.

At first assume that \(i(e)=i(e')\); then \(e\) and \(e'\) span a biangle. Moreover, by the induction hypothesis, we can assume that \(\gamma=\gamma'\).
Hence, it follows from Lemma~\ref{sec:gkm-graphs-2} that \(\pi(e)\) and \(\pi(e')\) have the same end points.

Next assume that \(i(e)\neq i(e')\). Then we may assume
\begin{align*}
  \gamma&=\gamma_1*\tilde{e}'& \gamma'&=\gamma_1*\tilde{e},
\end{align*}
where \(\gamma_1\) is some minimising path and \(\tilde{e}'\), \(\tilde{e}\) are parallel to \(e'\) and \(e\) respectively.
Hence \(\tilde{e}'\) and \(\tilde{e}\) span a square at the end point of \(\gamma_1\).
Hence, it follows from Lemma~\ref{sec:gkm-graphs-2} that \(\pi(\tilde{e}')\) and \(\pi(\tilde{e})\) also span a square in \(\Gamma\).
Therefore it follows that \(\pi(e)\) and \(\pi(e')\) have the same end points.

Finally we have to consider edges $e$ of $\tilde\Gamma_{n}$ such that both $i(e)$ and $t(e)$ are not vertices in $\tilde\Gamma_{n-1}$. In this case  there are shortest paths from $x_0$ to $i(e)$, as well as to $t(e)$, of the following form:
\begin{align*}
  \gamma_1&*e_1&  \gamma_1 &*f_1, 
\end{align*}
where the edges in the two paths satisfy the same relations as above and \(e_1\) and \(f_1\) are tangent to the same factor.
This is because if two vertices of \(\tilde\Gamma\) are connected by an edge, then they only differ in one coordinate. We have to show that the two possible definitions \(\pi(e)=\nabla_{\pi(\gamma_1*e_1)}\pi(\tilde{\nabla}_{\gamma_1*e_1}^{-1} e)\) and \(\pi(\bar{e})=\nabla_{\pi(\gamma_1*f_1)}\pi(\tilde{\nabla}_{\gamma_1*f_1}^{-1} \bar{e})\) for the image of $e$ are compatible.

Since the length of the two paths are the same and they are minimising, it follows that $e_1$ and $f_1$ span a triangle, with \(e\) as third edge.

Then we have, using the induction hypothesis,
\begin{align*}
\nabla_{\pi(\gamma_1 * e_1)} \pi(\tilde{\nabla}_{\gamma_1*e_1}^{-1} e) &= \nabla_{\pi(e_1)} \nabla_{\pi(\gamma_1)} \pi(\tilde\nabla_{\gamma_1}^{-1} \tilde\nabla_{e_1}^{-1}(e)) \\
&= \nabla_{\pi(e_1)} \nabla_{\pi(\gamma_1)}\pi(\tilde\nabla_{\gamma_1}^{-1} f_1) = \nabla_{\pi(e_1)} \pi(f_1),
\end{align*}
and analogously,
\[
\nabla_{\pi(\gamma_1*f_1)} \pi(\tilde\nabla_{\gamma_1*f_1}^{-1} \bar{e}) = \nabla_{\pi(f_1)} \pi(e_1).
\]
Thus we are done if we can show that $\pi(e_1)$ and $\pi(f_1)$ span a triangle in $\Gamma$. But this follows from Lemma \ref{sec:gkm-graphs-2}. 

We have thus shown that \(\pi:\tilde\Gamma\to \Gamma\) is a well-defined map of graphs. Next, we confirm that \(\pi\) is compatible with the connections \(\tilde{\nabla}\) and \(\nabla\).

By construction \(\pi\) has the following properties:
\begin{itemize}
\item \(\pi\) maps two-dimensional faces of \(\tilde{\Gamma}\) of a given type (\(I^2,\Sigma^2,\Delta^2\)) to a face of \(\Gamma\) of the same type.
\item If \(e\) is an edge of \(\tilde{\Gamma}\) which is part of a shortest path to the base point, then we have
  \begin{equation*}
    \nabla_{\pi(e)}\circ \pi =\pi\circ \tilde{\nabla}_e,
  \end{equation*}
i.e. \(\pi\) is compatible with \(\tilde{\nabla}_e\) and \(\nabla_{\pi(e)}\).
\end{itemize}

Hence, it only remains to be shown that \(\pi\) is also compatible with the connection of those edges which are not part of a minimising path to the base point.
These edges \(e\) are tangent to factors \(\Delta^n\), \(n\geq 2\) and are opposite to the closest point \(y_0\in\Delta^n\) to \(x_0\).
Then \(e\) together with two edges \(f_1,f_2\) which connect $y_0$ with the initial and end point of $e$ form a triangle \(\Delta^2_0\).
Note that \(f_1\) and \(f_2\) are part of minimising paths to \(x_0\).
Therefore for an edge \(e'\) starting at the same point as \(e\) and not tangent to \(\Delta^2_0\) we have:
\begin{equation*}
  \pi(\tilde{\nabla}_ee')=\pi(\tilde{\nabla}_{f_2}\circ \tilde{\nabla}_{\bar{f_1}} e')=\nabla_{\pi(f_2)}\circ \nabla_{\pi(\bar{f_1})}\circ \pi(e')=\nabla_{\pi(e)}\circ\pi(e'),
\end{equation*}
because \(\pi(e),\pi(f_1),\pi(f_2)\) form a triangle in \(\Gamma\).

Moreover, by the same reason, we have
\begin{align*}
  \pi(\tilde{\nabla}_e \bar{f_1})= \pi(\bar{f_1}) = \nabla_{\pi(e)}\pi(\bar{f_2}).
\end{align*}
Hence $\pi$ is compatible with the connections. Finally it follows that $\pi$ is a covering.
\end{proof}

This theorem directly implies Theorem \ref{sec:introduction}:
\begin{proof}[Proof of Theorem \ref{sec:introduction}]
The GKM graph of a GKM$_4$ orbifold with an invariant metric of nonnegative curvature is a graph with small three-dimensional faces. By Theorem \ref{thm:covering}, any such graph is finitely covered by the vertex-edge graph of a finite product of simplices.
\end{proof}

\begin{definition}
Let $\Gamma$ be a graph with small three-dimensional faces, $x\in V(\Gamma)$, and  $\pi:\tilde\Gamma =\prod_{i=1}^k \sigma_i\to \Gamma$ a covering as in Theorem \ref{thm:covering}. Then a deck transformation of $\pi$ is an automorphism $\psi:\tilde\Gamma\to \tilde\Gamma$ such that $\pi\circ \psi = \psi$ and $\nabla_{\psi(e)} \circ \psi = \psi\circ \tilde\nabla_e$ for all edges $e\in E(\tilde\Gamma)$.
\end{definition}
Clearly, the deck transformations of $\pi$ form a group.

\begin{proposition} The covering $\pi$ is Galois, i.e., the deck transformation group of $\pi$ acts simply transitively on the fibers of $\pi$.
\end{proposition}
\begin{proof}
Let $x,y\in V(\tilde\Gamma)$ such that $\pi(x)=\pi(y) = : z$. The covering $\pi$ induces a bijection
\[
f:E(\tilde\Gamma)_x\overset{\pi}{\longrightarrow} E(\Gamma)_{z} \overset{\pi^{-1}}{\longrightarrow} E(\tilde\Gamma)_y.
\]
We first claim that $f$ respects the combinatorial types of the local factors at $x$ and $y$. In fact, this property is clear for $\pi:E(\tilde\Gamma)_x\to E(\Gamma)_z$. At the vertex $y$, the covering $\pi$ necessarily maps a local factor attached to $y$ to a local factor of the same type. As the combinatorial structure of the local factors at $y$ is, by Lemma \ref{sec:gkm-graphs-2}, the same as that of the local factors at $x$, and hence also the same as that of the local factors at $z$, it follows that $\pi$ has to respect the combinatorial structures of the local factors at $y$ as well.

Thus Theorem \ref{thm:covering}, applied to $\Gamma:=\tilde\Gamma$, implies that $f$ extends uniquely to an automorphism $\psi:\tilde\Gamma\to \tilde\Gamma$ respecting the natural connection of $\tilde\Gamma$. By the uniqueness statement of Theorem \ref{thm:covering}, the maps $\pi$ and $\pi\circ \psi$ are identical, hence $\psi$ is a deck transformation.

This shows that the deck transformation group acts transitively on the fibers of $\pi$. The uniqueness statement of Theorem \ref{thm:covering}  implies that the action of the deck transformation group on the fibers is also free.
\end{proof}

Now let $\Gamma$ be the GKM graph of a nonnegatively curved GKM$_4$ manifold, and $\pi:\tilde\Gamma\to \Gamma$ as above. Using $\pi$, we pull back the labeling of $\Gamma$; in this way $\tilde\Gamma$ becomes a GKM$_4$ graph. By construction, the deck transformation group of $\pi$ leaves invariant the labeling of $\tilde\Gamma$. 

\section{Extending GKM graphs}
\label{sec:ext}

We say that a GKM graph is effective if at each vertex \(v\in \Gamma\), we have that \(\alpha(e_1),\dots,\alpha(e_k)\) generate \(\mathfrak{t}^*\), where \(e_1,\dots,e_k\) are the edges emanating from \(v\).

Note that every GKM graph can be made effective by replacing the torus \(T\) by \(T/\bigcap_{i=1}^k \exp(\ker \alpha(e_i))\).

\begin{definition}\label{defn:gkmextension}
Let $m\geq k$. We consider two effective GKM$_3$ graphs with the same underlying graph with connection, $(\Gamma,\alpha,\nabla)$ and $(\Gamma,\beta,\nabla)$, for an $k$- respectively $m$-dimensional torus $T^k$ respectively $T^{m}$. If there exists a linear map \(\phi:(\mathfrak{t}^m)^*\rightarrow (\mathfrak{t}^k)^*\), such that \(\alpha=\phi\circ \beta\), then we say that $(\Gamma,\beta,\nabla)$ is an extension of $(\Gamma,\alpha,\nabla)$.
\end{definition}
Recall that $\beta$ is only well-defined up to signs, i.e., is a map $\beta:E(\Gamma)\to (\mathfrak{t}^m)^*/\{\pm 1\}$. The composition $\phi\circ \beta:E(\Gamma)\to (\mathfrak{t}^k)^*/\{\pm 1\}$ is thus also well-defined only up to signs.

Because a linear map sends linearly dependent vectors to linearly dependent vectors, any extension of a GKM$_k$ graph is again GKM$_k$.

Let $(\Gamma,\beta)$ be an extension of an effective GKM graph $(\Gamma,\alpha)$ in the sense of Definition \ref{defn:gkmextension}. If $e_1,e_2$ are two edges that meet in a vertex $v$, then by \eqref{eq:1}
\begin{align*}
  \hat\beta(\nabla_{e_1}(e_2))&=q'\hat\beta(e_1)+p'\hat\beta(e_2),\\
  \hat\alpha(\nabla_{e_1}(e_2))&=q\hat\alpha(e_1)+p\hat\alpha(e_2),
\end{align*}
for some \(p,p',q,q'\in \Q\). If we have chosen the signs of $\hat\beta$ in such a way that \(\phi(\hat\beta(e_i))=\hat\alpha(e_i)\) and $\phi(\hat\beta(\nabla_{e_1}e_2)) = \hat\alpha(\nabla_{e_1}e_2)$, then we see by applying \(\phi\) to the first equation and by using the $2$-independence of the weights that $q=q'$ and $p=p'$. 

From this argument we get a necessary and sufficient condition for extending an \(n\)-valent GKM$_3$ graph \((\Gamma,\alpha)\) to a GKM$_n$ graph.
This goes as follows.

At first choose a basis \(b_1,\dots,b_n\) of \((\mathfrak{t}^n)^*\) and a base point \(v_0\) of \(\Gamma\). Let \(e_1,\dots,e_n\) be the edges with initial point \(v_0\).
We want to define \(\beta(e_i)=\pm b_i\).
Hence \(\phi\) must be defined by the equations \(\phi(b_i)=\hat\alpha(e_i)\), for some choices of signs $\hat\alpha(e_i)$.
As the weights on the edges emanating from \(v_0\) are defined, we next want to define the weights for the edges emanating from a vertex connected to $v_0$ by one edge, say \(e_1\).

As in the above computation for these edges \(\nabla_{e_1}(e_j)\),  the numbers $p,q\in \Q$ such that
\begin{equation*}
  \hat\alpha(\nabla_{e_1}(e_j))= q\hat\alpha(e_1)+p\hat\alpha(e_j)
\end{equation*}
(for an arbitrary choice of sign of $\hat\alpha(\nabla_{e_1}(e_j))$) also have to satisfy
\begin{equation*}
  \hat\beta(\nabla_{e_1}(e_j))=q\hat\beta(e_1)+ p\hat\beta(e_j),
\end{equation*}
where the sign of $\hat\beta(\nabla_{e_1}(e_j))$ is chosen such that $\phi(\hat\beta(\nabla_{e_1}(e_j)) = \hat\alpha(\nabla_{e_1}(e_j))$. 
Hence the weights $\beta(\nabla_{e_i}e_j)$ are uniquely determined by the weights of the edges at \(v_0\).

Next one can consider edges emanating from vertices which are two edges away from \(v_0\).
By the above argument these are uniquely determined by the weights of edges at vertices which are one edge away from \(v_0\).
Iterating this argument, one sees that if one fixes the weights at the vertex \(v_0\), there is at most one way to define the weights of the other edges.

One can define these edges consistently if and only if weights at a vertex \(v_1\) are independent of the paths from \(v_0\) to \(v_1\) which is used to transport the weights from \(v_0\) to \(v_1\).
That is, we can define the extension if and only if transporting the weights \(\beta(e_i)\)  at \(v_0\) around a loop based at \(v_0\) in a way prescribed by the weights $\alpha$ leads to the same weights \(\beta(e_i)\) at \(v_0\). Formally, we fix an arbitrary lift $\hat\alpha$ of $\alpha$, and consider arbitrary paths $\gamma$ based at $v_0$ of the form 
\[
f_1 * \dots * f_l,
\]
as well as another edge $e$ at $v_0$. We define, inductively on $j=1,\ldots,l$, the weight of the edge $\nabla_{f_1*\dots *f_j}e$ which is obtained by transporting $\hat\beta(e)$ along the path $f_1*\dots * f_j$ in a way prescribed by the weights $\alpha$: for $j=1$, we put
\[
A^\gamma_{f_1} (e) := p\hat\beta(e) + q\hat\beta(f_1),
\]
where $p$ and $q$ are determined by the equation
\[
\hat\alpha(\nabla_{f_1}e) = p \hat\alpha(e) + q\hat\alpha(f_1).
\]
If $A^\gamma_{f_1*\dots * f_{j-1}}(e)$ is already defined for all edges $e$ at $v_0$, then we set
\begin{equation} \label{eq:extension2}
A^\gamma_{f_1*\dots * f_j}(e):= p A^\gamma_{f_1*\dots * f_{j-1}}(e) + q A^\gamma_{f_1*\dots * f_{j-1}}(\nabla_{f_1*\dots * f_{j-1}}^{-1} f_j),
\end{equation}
where $p$ and $q$ are given by
\[
\hat\alpha(\nabla_{f_1*\dots * f_j}e) = p\hat\alpha(\nabla_{f_1*\dots * f_{j-1}}e) + q \hat\alpha(f_j).
\]
Then, the prescribed $\beta(e_i)$ can be completed to a well-defined extension of $(\Gamma,\alpha)$ to a GKM$_n$-graph if and only if
\begin{equation}\label{eq:extension1}
A^\gamma_\gamma (e) = \pm \hat\beta(\nabla_\gamma e)
\end{equation}
for all closed loops $\gamma$ based at $v_0$. (Note that $\nabla_\gamma e$  is an edge at $v_0$.)

This reasoning leads to the following lemma:

\begin{lemma}
\label{sec:gkm-graphs-5}
  Let \((\Gamma,\alpha)\) be an \(n\)-valent GKM$_4$ graph which is combinatorially equivalent to a product \(\prod_i \Delta^{n_i}\times \prod_i\Sigma^{m_i}\).
  Then \((\Gamma,\alpha)\) extends to an GKM$_n$ graph.
  This extension is uniquely determined by the weights at a single vertex in \(\Gamma\).
  \end{lemma}
\begin{proof}
Consider first the situation of a loop $\gamma$ at $v_0$ containing two successive edges that are tangent to different factors of the product \(\prod_i \Delta^{n_i}\times \prod_i\Sigma^{m_i}\), which hence span a square. Denoting by $\gamma'$ the loop obtained from $\gamma$ by replacing the two edges by its opposite edges, we note that as in the proof of Theorem \ref{thm:covering} the connection actions of $\gamma$ and $\gamma'$ coincide. We wish to show that the relations \eqref{eq:extension1} are equivalent for $\gamma$ and for $\gamma'$. Concretely, we write
\[
\gamma = \gamma_1 * f_1 * f_2 * \gamma_2,
\]
where $f_1$ and $f_2$ are two edges that span a square, with opposite edges $f_1'$ and $f_2'$, and put
\[
\gamma' = \gamma_1 * f_2' * f_1' * \gamma_2.
\]
Then, for any edge $e$ at $v_0$,
\begin{align*}
A^\gamma_{\gamma_1*f_1*f_2}(e) &= p_1A^\gamma_{\gamma_1*f_1}(e) + q_1A^\gamma_{\gamma_1*f_1}(\nabla_{\gamma_1*f_1}^{-1} f_2)\\
&=p_1(p_2 A^\gamma_{\gamma_1} (e) + q_2 A^\gamma_{\gamma_1}(\nabla_{\gamma_1}^{-1}f_1)) + q_1 A^\gamma_{\gamma_1*f_1}(\nabla_{\gamma_1}^{-1} f_2')\\
&= p_1(p_2 A^\gamma_{\gamma_1} (e) + q_2 A^\gamma_{\gamma_1}(\nabla_{\gamma_1}^{-1}f_1)) + q_1 (p_3 A^\gamma_{\gamma_1}(\nabla_{\gamma_1}^{-1} f_2') + q_3 A^\gamma_{\gamma_1}(\nabla_{\gamma_1}^{-1}f_1))
\end{align*}
for some rational numbers $p_i$ and $q_i$. Thus, $A^\gamma_{\gamma_1*f_1*f_2}(e)$ is contained in the three-dimensional span $V$ of $A^\gamma_{\gamma_1}(e)=A^{\gamma'}_{\gamma_1}(e)$, $A^\gamma_{\gamma_1}(\nabla_{\gamma_1}^{-1}f_1)=A^{\gamma'}_{\gamma_1}(\nabla_{\gamma_1}^{-1}f_1)$ and $A^\gamma_{\gamma_1}(\nabla_{\gamma_1}^{-1} f_2')=A^{\gamma'}_{\gamma_1}(\nabla_{\gamma_1}^{-1} f_2')$. The same computation shows that  $A^{\gamma'}_{\gamma_1*f_2'*f_1'}$ is contained in the same space $V$. Now as by construction the restriction map $\phi$ sends the above basis of $V$ to $\hat\alpha(\nabla_{\gamma_1}e)$, $\hat\alpha(f_1)$ and $\hat\alpha(f_2')$, it is injective on $V$. As $\phi(A^\gamma_{\gamma_1*f_1*f_2}(e)) = \pm \hat\alpha(\nabla_{\gamma_1*f_1*f_2} e) = \phi(A^{\gamma'}_{\gamma_1*f_1'*f'_2}(e))$, we conclude that $A^\gamma_{\gamma_1*f_1*f_2}(e) = \pm A^{\gamma'}_{\gamma_1*f_1'*f_2'}(e)$ for all edges $e$ at $v_0$. 

Then, by successively arguing along the edges of $\gamma_2$ using Equation \eqref{eq:extension2}, we conclude that $A^\gamma_{\gamma}(e) = A^{\gamma'}_{\gamma'}(e)$ for all edges $e$ at $v_0$, which shows the claim.

Hence, it remains to show Equation \eqref{eq:extension1} for loops \(\gamma\) of the from $\gamma_1*\dots*\gamma_k$, 
where \(\gamma_i\) is a closed path in the \(i\)-th factor \(\Gamma_i\) of \(\Gamma\). However for this it suffices to prove the equation for loops $\gamma$ in one of the factors $\Gamma_i$. 

Consider the case that $\Gamma_i$ is a simplex $\Delta^{n_i}$. Note that if a loop $\gamma$ is of the form $\gamma_1 * f_1 * \bar{f_1} * \gamma_2$, then Equation \eqref{eq:extension1} holds for $\gamma$ if and only if it holds for the loop $\gamma_1*\gamma_2$. As any loop in a simplex can, up to insertion of paths of the form $f*\bar{f}$, be written as a composition of boundaries of triangles, it suffices to show \eqref{eq:extension1} for $\gamma=f_1*f_2*f_3$ the boundary of a triangle. For such a loop one computes, analogously to the computation above, that $A^\gamma_\gamma(e)$, for any edge $e$ at $v_0$, is contained in the linear span $V$ of  $\hat\beta(e)$, $\hat\beta(f_1)$ and $\hat\beta(f_3)$, and is sent to $\pm \hat\alpha(\nabla_\gamma e)$ by $\phi$. On the other hand, $\hat\beta(\nabla_\gamma e)$ restricts to $\hat\alpha(\nabla_\gamma e)$, and is also contained in $V$ (this is because if $e$ is any edge different from $f_1$ or $f_3$, then $\nabla_\gamma e=e$; if $e=f_1$, then $\nabla_\gamma e = \bar{f_3}$, and if $e=\bar{f_3}$, then $\nabla_\gamma e = f_1$). As $\phi$ is injective on $V$, we conclude that $A^\gamma_\gamma(e)=\pm\hat\beta(\nabla_\gamma e)$.

A similar argument, only easier, goes through for the factors of type $\Sigma^{m_i}$. This concludes the proof.
\end{proof}

\section{A model}
\label{sec:model}

In this section we prove the following theorem which implies Theorem~\ref{sec:introduction-3}:

\begin{theorem}\label{thm:amodel}
  Let \(O\) be a GKM$_4$ orbifold which is rationally elliptic or non-negatively curved. Then there is a torus orbifold \(\tilde{O}\) and an action of a finite group \(G\) on \(\tilde{O}\) which normalises the torus action on \(\tilde{O}\) such that
  \begin{equation*}
    H^*(O;\mathbb{Q})\cong H^*(\tilde{O}/G;\mathbb{Q}).
  \end{equation*}
Moreover, \(\tilde{O}\) admits a metric of non-negative curvature, which is invariant both under the torus action and the \(G\)-action.

If \(O\) is a manifold, then also \(\tilde{O}\) is a manifold. 
\end{theorem}

For the proof of the theorem we need several lemmas. Throughout this section we use the following notation.
We denote the ($k$-dimensional) torus acting on $O$ by $T^k$, the GKM graph of \(O\) by \((\Gamma,\alpha)\) and the covering of \(\Gamma\) provided by Theorem \ref{thm:covering} by \(\pi:\tilde\Gamma\to \Gamma\). We let \(G\) be the deck transformation group of this covering. Denoting the lifted labeling on $\tilde\Gamma$ by $\tilde\alpha$, we can extend $(\tilde\Gamma,\tilde\alpha)$ by Lemma \ref{sec:gkm-graphs-5} to a GKM$_n$ graph, for an $n$-dimensional torus $T^n$. We denote its weights by \(\beta\), and we have a linear map $\phi:(\mathfrak{t}^n)^*\to (\mathfrak{t}^k)^*$ such that $\alpha = \phi\circ \beta$.

Moreover, by \(P\) we denote the product \(\prod_i \Delta^{n_i}\times \prod_i \Sigma^{m_i}\) of which \(\tilde\Gamma\) is the vertex-edge-graph.

\begin{lemma}
\label{sec:conjecture-1}
  In the above situation \(G\) acts by face-preserving homeomorphisms on \(P\).
\end{lemma}
\begin{proof}
  We will show that \(G\) acts by automorphisms on the face poset of \(P\). The action on \(P\) can then be constructed inductively because all the faces of \(P\) are homeomorphic to cones over their respective boundaries.

The vertex-edge graphs of the \(k\)-dimensional faces of \(P\) are the \(\nabla\)-invariant \(k\)-valent subgraphs of \(\tilde\Gamma\). Since the \(G\)-action on \(\tilde\Gamma\) is compatible with \(\nabla\), it leaves the set of these subgraphs invariant. Hence, we have a \(G\)-action on the face poset of \(P\) and the claim follows.
\end{proof}

Let \(I^*\subset {\mathfrak{t}^n}^*\) be the integer lattice spanned by all the \(\beta(e)\) where \(e\) runs through the edges of \(\tilde\Gamma\). Moreover denote by \(I\subset \mathfrak{t}^n\) the dual lattice and by \(\check{T}^n\) the quotient torus \(\mathfrak{t}^n/I\). Note that \(T^n\) finitely covers \(\check{T}^n\).

Note also that if \(O\) is a manifold, then \(I^*\) is spanned by the weights of the edges meeting in one vertex, by the fact that in this case we can choose $p=\pm 1$ and $q$ to be an integer in Equation (\ref{eq:1}).

\begin{lemma}
\label{sec:conjecture}
In the above situation \(G\) acts by automorphisms on the torus \(\check{T}^n\) such that for all \(g\in G\), \(e\in E(\tilde\Gamma)\):
  \begin{equation*}
    \beta(ge)=g \beta(e)
  \end{equation*}
\end{lemma}
\begin{proof}
  Let \[
  H=\{(g,f)\in G\times \Aut(\check{T}^n);\;\beta(ge)= f\circ\beta(e) \text{ for all } e\in E(\tilde\Gamma) \text{ and } \phi\circ f =\phi\}.
  \]
  Recall that $\beta$ takes values in ${\mathfrak{t}^n}^*/\{\pm 1\}$.
We claim that the projection on the first factor of \(H\) is an isomorphism.
We first show surjectivity. Let \(g\in G\).

Let \(v_0\) be a vertex of \(\tilde\Gamma\) and \(e_1,\dots,e_n\) be the edges of \(\tilde\Gamma\)  meeting at \(v_0\). Then, after choosing signs, the \(\hat\beta(e_1),\dots,\hat\beta(e_n)\) form a basis of \(\mathfrak{t}^*\).
Moreover, we can choose signs for the \(\hat\beta(ge_1),\dots,\hat\beta(ge_n)\) in such a way that
\[\phi(\hat\beta(e_i))=\hat\alpha(\pi(e_i))=\phi(\hat\beta(ge_i))\] holds for all \(i\).
Therefore, for each \(g\in G\), the automorphism \(A_g\) of \({\mathfrak{t}^n}^*\) defined by 
\begin{equation}
\label{eq:2}
  A_g\hat\beta(e_i)=\hat\beta(ge_i)
\end{equation}
satisfies \(\phi\circ A_g=\phi\). Now by
\begin{align*}
  e\longmapsto \beta(e)\qquad \text{ and } \qquad e\longmapsto A_g^{-1}\beta(ge)
\end{align*}
there are two GKM$_n$ extensions of \(\tilde\alpha\) defined which agree on edges at the vertex \(v_0\).
Therefore by Lemma \ref{sec:gkm-graphs-5} they agree everywhere.
Hence, \(A_g\) maps the lattice \(I^*\) isomorphically to itself and therefore its dual $A_g^*$ defines an element of \(\Aut(\check{T}^n)\). Thus surjectivity is proven.

To prove injectivity, let \((\id,f)\in\ker(H\rightarrow G)\).
Then for \(i=1,\dots,n\) we have
\begin{equation*}
  f\circ\beta(e_i)=\beta(e_i).
\end{equation*}
Since \(\phi\circ\beta(e_i)\neq 0\), it follows from the requirement that \(\phi\circ f=\phi\) that \(f\) must be the identity and the claim follows.
\end{proof}

By the above lemma we have a \(G\)-action by automorphism on \(\check{T}^n\). 
For later reference we  note here also the following

\begin{lemma}\label{sec:gkm-manifolds-with}
In the above situation the map $\phi^*:{\mathfrak{t}}^k\to {\mathfrak{t}}^n$ descends to a homomorphism $\phi^*:T^k\to \check{T}^n$.  The above \(G\)-action restricts to the trivial action on \(\check{T}^{k}=\phi^*(T^k)\subset \check{T}^n\), where the torus \(T^k\) is associated to the covered GKM graph \(\Gamma\).
\end{lemma}
\begin{proof}
As $\phi\circ\beta = \alpha$, the map $\phi$ sends $I^*$ onto the lattice $J^*$ in ${{\mathfrak{t}}^k}^*$ spanned by all the $\alpha(e)$. Thus, $\phi^*$ sends the dual lattice $J$ injectively to $I$. As $T^k = {{\mathfrak{t}}^k}/J$, it follows that $\phi^*$ induces a well-defined map on $T^k$. 

In the notation of the proof of Lemma \ref{sec:conjecture} we have $\phi^* = A_g\circ \phi^*$ for all $g\in G$.
\end{proof}

By dualising the weights \(\beta\) we get a labeling \(\lambda\) of the facets of \(P\) by one-dimensional subgroups of \(\check{T}^n\) which is also compatible with the \(G\)-actions on \(P\) and \(\check{T}^n\) by the above lemmas.

Therefore there is a \emph{continuous} \(G\)-action on the torus orbifold
\begin{equation*}
  \tilde{O}=(P\times \check{T}^n)/\sim 
\end{equation*}
where \((x_1,t_1)\sim (x_2,t_2)\) if and only if \(x_1=x_2\) and \(t_1t_2^{-1}\in \langle \lambda(F);\; x_1\in F\rangle\),
which normalises the \(\check{T}^n\)-action.

Note that in the case that \(O\) is a manifold, \(\tilde{O}\) is also a manifold.
This is because in this case the weights at every vertex \(v\in P\) form a basis of \(I^*\).
Hence in this case \(\check{T}^n=\prod_{v\in F}\lambda(F)\) for all vertices \(v\in P\).
This implies that \(\tilde{O}\) is locally modelled on effective complex \(n\)-dimensional \(\check{T}^n\)-representations.
So it is a manifold.

Moreover, since \(G\) acts trivially on \(T^k\), we have
\begin{align*}
  H^*(O)&=H^*_{T^k}(O)/(H^2(BT^k))\\ &=H^*_{T^k}(\Gamma)/(H^2(BT^k)) &&\text{by the GKM Theorem}\\
&=H^*_{\check{T}^k}(\Gamma)/(H^2(B\check{T}^k))\\
  &=H^*_{\check{T}^k}(\tilde\Gamma)^G/(H^2(B\check{T}^k)) &&\text{because } \Gamma=\tilde{\Gamma}/G \\ &=H^*_{\check{T}^k}(\tilde{O})^G/(H^2(B\check{T}^k)) &&\text{by the GKM Theorem again}\\
&=(H^*_{\check{T}^k}(\tilde{O})/(H^2(B\check{T}^k)))^G\\ &=H^*(\tilde{O})^G=H^*(\tilde{O}/G).
\end{align*}
Hence we have proven Theorem \ref{thm:amodel} except for the claim about the invariant non-negatively curved metric on $\tilde O$.

This will follow from the next lemma because the orbit space of an isometric group action on a non-negatively curved manifold equipped with the quotient metric is a non-negatively curved Alexandrov space.
\begin{lemma}
  Let \(Z_P=\prod_i S^{2n_i+1}\times \prod_i S^{2m_i}\) be the moment angle complex associated to \(P=\prod_i\Delta^{n_i}\times \prod_i \Sigma^{m_i}\) equipped with the natural product metric.
  Then the \(G\)-action on \(\tilde{O}\) lifts to an isometric action on \(Z_P\).
\end{lemma}
\begin{proof}
  Let \(F_1,\dots,F_m\) be the facets of \(P\). Then the moment angle complex associated to \(P\) is given by
  \begin{equation*}
    Z_P=P\times T^m/\sim,
  \end{equation*}
where \(T^m=\prod_{i=1}^m S^1_i\) and \((x_1,t_1)\sim (x_2,t_2)\) if and only if \(x_1=x_2\) and \[t_1t_2^{-1}\in \prod_{i;\; x_1\in F_i}S_i^1.\]

We have a map \(\varphi:T^m\rightarrow \check{T}^n\) such that the restriction of \(\varphi\) to \(S^1_i\) is an isomorphism \(S^1_i\rightarrow \lambda(F_i)\). With this map we can realise \(\tilde{O}\) as a quotient of an almost free action of an abelian Lie group on \(Z_P\).

We can lift the \(G\)-action on \(\tilde{O}\) to a  \(G\)-action on \(Z_P\) by requiring that for \(g\in G\), \(g(S^1_i)=S^1_j\) with \(gF_i=F_j\) and that the following diagram commutes:

\begin{equation*}
  \xymatrix{S^1_i\ar[r]_g\ar[d]_\varphi&S^1_j\ar[d]_\varphi\\
\lambda(F_i)\ar[r]_g&\lambda(F_j)}
\end{equation*}

Now note that in the case that \(P\) is a product of \(\Delta^{n_i}\) and \(\Sigma^{m_i}\), \(Z_P\) is equivariantly diffeomorphic to a product of round spheres of radius one.
Moreover, note that the combinatorial type of the \(G\)-action on \(Z_P\), that is, the action of \(G\) on \(\mathcal{P}(P)\times T^m\) can be realised by elements of the Weyl-group of \(\left(\prod_i O(2n_i+2)\times \prod_i O(2m_i+1)\right)\rtimes H\), where \(H\) is the group of permutations of the factors of the same dimension in the product of spheres.
Here \(\mathcal{P}(P)\) denotes the face poset of \(P\) as defined in Section~\ref{sec:torus-manif-orbif}.
Therefore the claim follows. 
\end{proof}

\begin{proof}[Proof of Theorem \ref{sec:introduction-1}]
In the special case that the graph \(\Gamma\) is the vertex-edge graph of a product \(\prod_{i}\Delta^{n_i}\times \prod_j\Sigma^{m_i}\), the covering of graphs \(\tilde\Gamma\rightarrow \Gamma\) is trivial.
Hence the deck transformation group \(G\) is the trivial group.
From the proof of Theorem~\ref{thm:amodel} it then follows that the group \(G\) in that theorem is trivial.
 Therefore Theorem~\ref{sec:introduction-1} follows from Theorem~\ref{thm:amodel}. Note that for this theorem no assumptions on rational ellipticity or non-negative curvature are needed, as they are only used to obtain information on the structure of the GKM graph; here we instead assume the graph to be of the simplest possible type.
\end{proof}

\section{Upper bound for the dimension of the acting torus}
\label{sec:upper-bound-dimens}

In this section we give an upper bound for the dimension of a torus which can act on a non-negatively curved GKM$_4$ orbifold \(O\).
Moreover, we give an example of a non-negatively curved GKM$_4$-manifold whose rational cohomology ring is not isomorphic to that of a torus manifold.
This shows that in  Theorem \ref{sec:introduction-3} the group \(G\) is not always trivial.

We denote by \(T\) the torus acting on \(O\). Otherwise we use the same notation as in the previous section.

\begin{theorem}
  \label{sec:upper-bound-dimens-4}
  Let \(O\) be a GKM$_4$-orbifold and \(T\) be the torus acting on \(O\).
  
  Then the dimension of a maximal torus \(T'\) to which the \(T\)-action on \(O\) can be extended is bounded from above by
\[a-b_2(O),\]
where \(a\) is the number of orbits of the \(G\)-action on the set of facets of \(P\).
\end{theorem}
\begin{proof}
  \(T'\) is a subtorus of \((T^n)^G\), where \(G\) acts on \(T^n\) as in Lemma \ref{sec:conjecture}.
  We have a short exact sequence of \(G\)-representations
  \begin{equation*}
    0\rightarrow H^2(BT^n)\rightarrow H^2_{T^n}(\tilde{O})\rightarrow H^2(\tilde{O})\rightarrow 0.
  \end{equation*}
 \(H^2_{T^n}(\tilde{O})\) has a basis \(v_1,\dots,v_n\)  consisting out of the Poincar\'e duals of the facial suborbifolds of \(\tilde{O}\).
 Since the \(G\)-action sends facial suborbifolds to facial submanifolds, the \(\pm v_i\) are permuted by the \(G\)-action.
Hence we have \(\dim H^2_{T^n}(\tilde{O})^G\leq a\). Since \(b_2(O)=\dim H^2(\tilde{O})^G\), it follows that
\begin{equation*}
  \dim T'\leq \dim H^2(BT^n)^G=\dim H^2_{T^n}(\tilde{O})^G - \dim H^2(\tilde{O})^G\leq a-b_2(O).
\end{equation*}
\end{proof}

Next we want to give an example of a GKM$_{n-1}$-manifold of dimension \(2n\) where the action does not extend to an effective action of an \(n\)-dimensional torus.

\begin{example}
  \label{sec:upper-bound-dimens-3}
Let \(P\) be \(I^n=[-1,+1]^n\) with facets \(F_{1,\pm 1},\dots, F_{n,\pm 1}\) such that \(F_{i,+1}\) and \(F_{i,-1}\) belong to the same factor.

Let \(e_1,\dots,e_n\) be the standard basis of \(\mathbb{R}^n\) and set
\(\lambda(F_{n,\pm 1})=\pm e_n\), \(\lambda(F_{i,\pm 1})=e_i\pm e_n\) for \(i=1,\dots,n-1\).
Then the pair \((P,\lambda)\) defines a quasitoric manifold \(M\) over \(P\) by
\( M=(P\times T)/\sim
\),
where \(T=\mathbb{R}^n/\mathbb{Z}^n\) and \((x_1,t_1)\sim (x_2,t_2)\) if and only if \(x_1=x_2\) and \[t_1-t_2\in (\bigoplus_{x_1\in F_{ij}} \mathbb{R}\lambda(F_{ij}))/\bigoplus_{x_1\in F_{ij}} \mathbb{Z}\lambda(F_{ij})).\]
The \(T\)-action on the second factor of \(P\times T\) induces the \(T\)-action on \(M\). The weights of the \(T\)-representation at the fixed point \(F_{1,\epsilon_1}\cap\dots\cap F_{n,\epsilon_n}\) are given by
\(
e_1^*,\dots,e^*_{n-1},\epsilon_ne_n^*-\epsilon_n\sum_{i\neq n} \epsilon_i e_i^*
\).

The map
\begin{align*}
  P\times T&\rightarrow P\times T& (x,\sum_{i}\alpha_ie_i)&\mapsto (-x,\sum_{i\neq n}\alpha_i e_i-\alpha_ne_n+\frac{1}{2}e_1)
\end{align*}
induces a free involution \(\tau\) on \(M\) which is \(T\)-equivariant with respect to the involution
\begin{align*}
  \phi:T&\rightarrow T&\sum_i\alpha_ie_i&\mapsto \sum_{i\neq n}\alpha_i e_i-\alpha_n e_n.
\end{align*}
Note, that \(\tau\) is orientation preserving if and only if \(n\) is odd.

Hence, \(\tau\) commutes with the action of the torus \(T'=\mathbb{R}^{n-1}/\mathbb{Z}^{n-1}\) and \(N=M/\tau\) becomes a GKM$_{n-1}$-manifold with the action of \(T'\).
Lemma \ref{sec:upper-bound-dimens-1} below shows that \(T'\) is a maximal torus acting on \(N\).
Moreover, Corollary \ref{sec:upper-bound-dimens-2} shows that \(N\) does not have the rational cohomology of a torus manifold.
\end{example}

\begin{lemma}
  \label{sec:upper-bound-dimens-1}
If, in the situation of Example~\ref{sec:upper-bound-dimens-3}, \(n\geq 3\) then we have \(b_2(N)=1\) and there is an element \(v\in H^2(N)\) with \(v^2\neq 0\). In particular, for all non-zero  \(v'\in H^2(N)\) we have \(v'^2\neq 0\).
\end{lemma}
\begin{proof}
  First note that \(H^2(N;\mathbb{Q})\cong H^2(M;\mathbb{Q})^{\mathbb{Z}/2}\).
  Hence it suffices to describe the \(\mathbb{Z}/2\)-action on \(H^2(M)\).

  We have a short exact sequence of \(\mathbb{Z}/2\)-representations
  \begin{equation*}
    0\rightarrow H^2(BT)\rightarrow H^2_T(M)\rightarrow H^2(M)\rightarrow 0.
  \end{equation*}
  Hence it suffices to describe the \(\mathbb{Z}/2\)-actions on \(H^2(BT)\) and \(H^2_T(M)\).
  The first action is induced by \(\phi\). Hence we have \(\dim H^2(BT)^{\mathbb{Z}/2}=n-1\).
  The action on \(H^2_T(M)\) can be described as follows. We have an isomorphism
  \begin{equation*}
    H^2_T(M)\cong \bigoplus_{i=1}^n(\mathbb{Q}v_{i,1}\oplus \mathbb{Q}v_{i,-1}),
  \end{equation*}
Here, after possible changes of the orientations of the facial submanifolds, the \(\mathbb{Z}_2\) action is given by \(v_{i,\epsilon}\mapsto v_{i,-\epsilon}\) for \(i=1,\dots,n\) because \(\tau\) interchanges opposite facial submanifolds of \(M\) and the \(v_{i,\epsilon}\) are the equivariant Poincar\'e duals of these manifolds.

Therefore we have \(\dim H^2_T(M)^{\mathbb{Z}/2}=n\) and the first claim follows.

To see the second claim note that \(v_{n,-1}+v_{n,+1}\in H^2_T(M)^{\mathbb{Z}/2\mathbb{Z}}\) maps to a non-zero element \(v\in H^2(N)\cong H^2(M)^{\mathbb{Z}/2\mathbb{Z}}\).
Moreover, by \cite[Theorem 4.14]{DavisJanuszkiewicz}, there are the following relations in \(H^*(M)\):
\begin{align*}
  [v_{i,1}]&=-[v_{i,-1}]&&\text{for }i=1,\dots,n-1\\
  [v_{n,1}]&=[v_{n,-1}]-\sum_{i=1}^{n-1}([v_{i,1}]-[v_{i,-1}])\\
  [v_{i,1}][v_{i,-1}]&=0&&\text{for }i=1,\dots,n\\
  \prod_{i=1}^n[v_{i,\pm 1}]&\neq 0.
\end{align*}
Therefore we see
\begin{align*}
  v^2[v_{n,1}]^{n-2}&=[v_{n,1}]^n=[v_{n,1}](-2)^{n-1}(\sum_{i=1}^{n-1}[v_{i,1}])^{n-1}\\
  &=(-2)^{n-1}(n-1)!\prod_{i=1}^{n}[v_{i,1}]\neq 0 \in H^{2n}(M).
\end{align*}
Hence we have \(v^2\neq 0\).
\end{proof}

\begin{cor}
  \label{sec:upper-bound-dimens-2}
  Assume \(n\geq 5\).
  Then the manifold \(N\) from Example \ref{sec:upper-bound-dimens-3} does not have the rational cohomology of a torus manifold.
\end{cor}
\begin{proof}
  Assume that there is a torus manifold \(O\) with the same rational cohomology as \(N\).
  Then because the total Betti number \(2^{n-1}\) of \(N\) is equal to the number of vertices in the GKM graph \(\Gamma\) of \(O\), the covering graph \(\tilde{\Gamma}\) must be the vertex-edge graph of \(I^n\) or \(I^{n-2}\times \Sigma^2\).

  In the second case the number of vertices in \(\Gamma\) and \(\tilde{\Gamma}\) are equal. Therefore the covering must be trivial so that we have \(\tilde{\Gamma}=\Gamma\).
    Therefore \(H^2_T(O;\mathbb{Q})\) has dimension \(2n-2\).
    From the exactness of
    \begin{equation*}
      0\rightarrow H^2(BT^n)\rightarrow H^2_T(O)\rightarrow H^2(O)\rightarrow 0,
    \end{equation*}
    we now get a contradiction to Lemma~\ref{sec:upper-bound-dimens-1}.

    Therefore assume that we are in the first case.
    Then the order of the deck transformation group \(G\)  of the covering \(\tilde{\Gamma}\rightarrow \Gamma\) is two. Let \(g\in G\) be the non-trivial element.
    Let \(\tilde{O}\) be the torus manifold constructed from \(O\) in Theorem~\ref{thm:amodel}.
    Denote by \(v_{i,\pm 1}\), \(i=1,\dots,n\), the equivariant Poincar\'e duals of the facial submanifolds  \(\tilde{O}_{1,\pm 1},\dots,\tilde{O}_{1, \pm 1}\) of \(\tilde{O}\).
    Then as in the proof of Theorem~\ref{sec:upper-bound-dimens-4}, we see that for each \(i\in \{1,\dots,n\}\) there is an \(j_i\in \{1,\dots,n\}\) and \(\epsilon\in \pm 1\) such that \(g^*v_{i,\pm 1}=\pm v_{j_i,\pm \epsilon_i}\).
    Since \(g\) has order two, we can orient the facial submanifolds in such a way that the minus sign only appears when \(j_i=i\) and \(\epsilon_i=1\).

    However, because the \(G\)-action commutes with the torus action, \(g\) preserves the orientation of the normal bundle of \(\tilde{O}_{i,\pm 1}\).
    Hence we also have a \(+\)-sign in this case.
    Moreover, by the same reason, the labeling of the facets of \(I^n\) is \(G\)-invariant.
    Because the labels of intersecting facets are linearly independent, it follows that \(j_i=i\) for all \(i\).

    Form the exactness of
    \begin{equation*}
      0\rightarrow H^*(BT^n)\rightarrow H^*_T(\tilde{O})^G\rightarrow H^2(O)\rightarrow 0
    \end{equation*}
    and Lemma~\ref{sec:upper-bound-dimens-1} it follows, that the \(G\)-action on the set of facets of \(I^n\) has two fixed points and \(n-1\) non-trivial orbits.

    Assume that the \(v_{i,\pm 1}\), \(i=1,\dots,n-1\) are not fixed by \(g\).
    Let \(\lambda_{i,\pm 1}\in \mathbb{Z}^n\), \(i=1,\dots,n\), be primitive generators of the Lie algebra of the circle subgroups of \(T^n=\mathbb{R}^n/\mathbb{Z}^n\) that fix \(\tilde{O}_{i,\pm 1}\). Then for \(i=1,\dots,n-1\) we have \(\lambda_{i,1}=\lambda_{i,-1}\). Moreover, \(\lambda_{1,1},\dots,\lambda_{n,1}\) is a basis of \(\mathbb{Z}^n\) and \[\lambda_{n,-1}=\pm\lambda_{n,1}+\sum_{i=1}^{n-1}\mu_i\lambda_{i,1},\] with some integers \(\mu_i\).
    Therefore, by \cite[Theorem 4.14]{DavisJanuszkiewicz}, we have the following relations in \(H^*(\tilde{O})\):
    \begin{align*}
        [v_{i,1}]&=-[v_{i,-1}]-\mu_i[v_{n,-1}]&&\text{for }i=1,\dots,n-1\\
  [v_{n,1}]&=\pm[v_{n,-1}]\\
  [v_{i,1}][v_{i,-1}]&=0&&\text{for }i=1,\dots,n.
    \end{align*}
  In particular, we see that \([v_{n,1}]\neq 0\in H^2(\tilde{O})^{\mathbb{Z}/2\mathbb{Z}}\) and \([v_{n,1}]^2=0\).
    Therefore we have a contradiction to Lemma~\ref{sec:upper-bound-dimens-1}.
\end{proof}

\section{GKM manifolds with invariant almost complex structures}
\label{sec:gkm-manifolds-with-2}

In this section we let \(M^{2n}\) be a GKM$_4$ manifold which admits an invariant almost complex structure whose GKM graph \(\Gamma\)  has small three-dimensional faces.
Denote by \(\tilde\Gamma\) the covering graph of \(\Gamma\) constructed in Theorem~\ref{thm:covering} and \(G\) the deck transformation group.
We let \(\tilde{\alpha}\) be the weights of the edges of \(\tilde\Gamma\).
Moreover, \(\beta\) are the weights of the extension to an GKM$_n$ graph of \((\tilde\Gamma,\tilde\alpha)\).

We will show in this section that under the above conditions \(M\) has the same rational cohomology as a generalised Bott manifold.
To do so we first show that \((\tilde{\Gamma},\beta)\) is the GKM graph of a generalised Bott manifold.
Then in a second step we show that the covering \(\tilde{\Gamma}\rightarrow \Gamma\) is trivial.
Together these two steps imply the above claim, which is stated in Corollary~\ref{sec:gkm-manifolds-with-3}.

The combinatorial consequence of the existence of an almost complex structure is that the weights of the oriented edges of the GKM graph have preferred signs.
By going through the arguments one then also sees that the weights of the extensions constructed in Section 2 have preferred signs.

Denote by \(\hat\alpha\) and \(\hat\beta\) the weights \(\tilde\alpha\) and \(\beta\) with the preferred sign, respectively.

We will show that in this case the GKM graph of \(M\) is a product of simplices.
In particular, by the following theorem, \(M\) has the same rational cohomology as a generalised Bott manifold.

\begin{theorem}\label{thm:graphofBott}
With the notation as above we have that \((\tilde\Gamma,\beta)\) is the GKM graph of a generalised Bott manifold.
\end{theorem}
\begin{proof}
  By the argument in \cite[Lemma 5.6]{GoertschesWiemeler} there are no local factors in the GKM graph of $M$ with the combinatorial type of $\Sigma^k$. Therefore \((\tilde{\Gamma},\beta)\) is the GKM graph of a quasitoric manifold over a product of simplices \(P=\prod_i\Delta^{n_i}\).
  We embed each factor \(\Delta^{n_i}\) of \(P\) as a convex polytope in \(\R^{n_i}\).
  From these embeddings we get an embedding of \(P\) in \(\R^{n}=\prod_i\R^{n_i}\).

  By Theorem 6.4 of \cite{choi10:_quasit}, a quasitoric manifold over a product of simplices is equivariantly homeomorphic to a generalised Bott manifold if and only if it admits an invariant almost complex structure.
  
By \cite{MR2603265} (see also Theorem 7.3.24 in \cite{BuchstaberPanov2}), a quasitoric manifold admits an invariant almost complex structure if and only if the following product is independent of the vertex \(v\) we are looking at:
\begin{equation}
  \label{eq:4}
    \det \sigma_v \cdot \det A_v.
  \end{equation}
  Here we use the following notation:
  \begin{itemize}
  \item   \(\sigma_v\) is the matrix with columns the vectors in direction of the edges emanating from \(v\).
That means \(\sigma_v\) is the matrix with the column vectors \(v_i-v\) where \(v_1,\dots,v_n\) are the vertices of \(P\) connected to \(v\) by an edge of \(\tilde{\Gamma}\).
\item  \(A_v\) has as columns the labels \(\hat\beta(e)\) in the same order as in the first matrix.
  \end{itemize}
Note here that our matrix \(A_v\) is the inverse of the matrix appearing in \cite{MR2603265} with columns the labels \(\lambda(F)\)  of the facets \(F\) containing \(v\).

Therefore it suffices to show that the sign of the product (\ref{eq:4}) is independent of the vertex \(v\) of \(P\).
To do so let \(v_1\) and \(v_0\) be vertices connected by an edge.
Then \(v_1-v_0\) is tangent to one of the direct Euclidean factors of the product \(\mathbb{R}^n=\prod_{i} \mathbb{R}^{n_i}\), say the \(i_0\)-th.
Note that if we go from \(v_1\) to \(v_0\) the following happens:
\begin{itemize}
\item The column \(v_1-v_0\) in the first matrix changes sign.
\item To other columns of this matrix which correspond to vertices \(v_2\) of \(P\), such that \(v_2-v_0\) is tangent to the \(i_0\)-th factor of the above product, \(v_0-v_1\) is added.
  \item The colums of this matrix which correspond to vertices which differ in another coordinate are unchanged.
\item The column \(\hat\beta(v_1-v_0)\) in the second matrix changes sign.
\item To the other columns of this matrix a multiple of this weight is added.
\end{itemize}
Hence, in total the sign of the product does not change.
Since any GKM graph is connected the claim follows.
\end{proof}

Our next goal is to show that the covering \(\tilde\Gamma\rightarrow \Gamma\) is trivial. To do so we have to analyse the action of \(G\) on \(\tilde\Gamma\).
By Lemma~\ref{sec:conjecture-1}, \(G\) acts by automorphisms of the face poset of \(\tilde{\Gamma}\). 

Moreover, by Lemma~\ref{sec:conjecture}, \(G\) acts on \((\mathfrak{t}^n)^*\) in a way that is compatible with the weights \(\beta\).
In the situation that we have an \(T\)-invariant complex structure we can strengthen the result of that lemma.
This is the content of the next lemma.

\begin{lemma}
  There is an action \(\varphi: G\rightarrow \Aut((\mathfrak{t}^n)^*)\) such that
  \begin{equation*}
    \varphi(g)\hat\beta(e)=\hat\beta(ge)
  \end{equation*}
for all \(g\in G\), \(e\in E(\tilde{\Gamma})\).
\end{lemma}
\begin{proof}
By Lemma~\ref{sec:conjecture}, we only have to show that the \(G\)-action constructed there is compatible with the signs of the weights.

Note that by construction the \(G\)-action is compatible with the signs of the weights at the base point \(v_0\).
Moreover, it is compatible with the connection on \(\tilde{\Gamma}\).
Since the connection is compatible with the signs of the weights, it follows that the \(G\)-action is compatible with these signs.
\end{proof}

With this lemma we can show that there is a factor of \(\tilde{\Gamma}\) such that not all facets belonging to this factor are in the same \(G\)-orbit.

\begin{lemma}
\label{sec:gkm-manifolds-with-1}
  Let \((\tilde{\Gamma},\hat\beta)\) be the GKM graph of a generalised Bott manifold \(M\). Assume that the group \(G\) acts on this graph. Then by dualising the action we get an action on the characteristic pair \((\mathcal{P},\lambda)\)  associated to \(M\).
Here \(\mathcal{P}\) is the face poset of the orbit space of the \(T\)-action on  \(M\) and \(\lambda\) denotes its labeling by Lie algebras of isotropy groups.

Assume that for each factor of \(M/T\) all facets of this factor belong to the same \(G\)-orbit. Then the \(G\)-action on \(\mathfrak{t}\) does not have any non-zero fixed points.
\end{lemma}
\begin{proof}
  Since \(M\) is a generalised Bott manifold, we can order the factors \(\Delta_1,\dots,\Delta_n\) of \(M/T\) and filter \(V_1\subset V_2\subset\dots\subset V_n=\mathfrak{t}\)  the Lie algebra of \(T\) in such a way that
  \begin{itemize}
  \item \(\dim V_i=\dim \prod_{k\leq i} \Delta_k\)
  \item \(\lambda(F_i)\in V_i\) for each facet of \(M/T\) belonging to \(\Delta_i\).
  \end{itemize}
The above filtration of \(\mathfrak{t}\) is related to the fact that the reduced characteristic matrix of \(M\) can be assumed to be a upper triangular vector matrix (see Proposition 6.2 and Remark 6.3 of \cite{choi10:_quasit}). Here the characteristic matrix of \(M\) is the matrix which has as its columns the vectors \(\lambda(F)\), \(F\in\mathcal{F}(M/T)\). Moreover the reduced characteristic matrix has precisely one column \(\lambda(F_i)\) for each factor \(\Delta_i\) of \(M/T\) such that \(F_i\) is a facet of \(\Delta_i\) (see \cite{choi10:_quasit} for a more precise description).

After one more reordering of the facets we can assume that there is an
\(i_0\geq 1\) such  that \[\dim \langle \lambda(F_k);\; F_k \text{ is a facet of } \Delta_{i}\rangle=\dim \Delta_{i}\]
for all \(i\leq i_0\) and
\[\dim \langle \lambda(F_k);\; F_k \text{ is a facet of } \Delta_{i}\rangle=\dim \Delta_{i}+1\]
for \(i>i_0\).
So the factors \(\Delta_i\), \(i\leq i_0\) are precisely those factors for which the corresponding column in the reduced characteristic matrix does not have non-zero off-diagonal entries.

Then for each \(i\leq i_0\), \[\sum_{F_{ki}} \lambda(F_{ki})=0,\]
where the sum is over those facets \(F_{ki}\) which belong to \(\Delta_{i}\).
Note that this sum is non-zero for \(i>i_0\).

Since \(V_{i_0}\) is generated by the \(\lambda(F_{ki})\) with \(i\leq i_0\) it follows that \(V_{i_0}\) is invariant under the \(G\)-action and \(V_{i_0}^G=0\).
Hence the claim follows by considering the \(G\)-representation \(\mathfrak{t}/V_{i_0}\) and induction over the dimension of \(\mathfrak{t}\).
\end{proof}

\begin{cor}
  Let \(M\) be a GKM$_4$ manifold with an invariant metric of non-negative sectional curvature and an invariant almost complex structure with GKM graph \(\Gamma\).

Then there is a factor of \(\tilde{\Gamma}\) such that not all its facets belong to the same \(G\)-orbit.
\end{cor}
\begin{proof}
  This follows from Lemmas \ref{sec:gkm-manifolds-with} and \ref{sec:gkm-manifolds-with-1}.
\end{proof}

Now we can prove the following

\begin{theorem}\label{thm:coveringtrivial}
  The covering \(\tilde{\Gamma}\rightarrow \Gamma\) is trivial.
\end{theorem}
\begin{proof}
  Let \(\Delta\) be a factor of \(\tilde{\Gamma}\) such that not all facets belonging to this factor are in the same \(G\)-orbit.
  Let \(H\) be one of the \(G\)-orbits of the facets of \(\Delta\).
Then \(B_1=\bigcap_{F\in H} F\) is a \(G\)-invariant non-empty face of \(\tilde{\Gamma}\).
Moreover, the preimage in the generalised Bott manifold of this face is a generalised Bott manifold.
Hence, by the above discussion there is a factor of \(B_1\) such that not all facets of this factor belong to the same \(G\)-orbit.

Hence by induction we can construct a non-empty face fixed by \(G\).
This is a contradiction because \(G\) acts freely on the set of vertices of \(\tilde{\Gamma}\).
\end{proof}

As a consequence of the discussion in this section we get:

\begin{cor}
\label{sec:gkm-manifolds-with-3}
  Let \(M\) be a non-negatively curved or rationally elliptic GKM$_4$ manifold which admits an invariant almost complex structure.
  Then the rational cohomology ring of \(M\) is isomorphic to the rational cohomology ring of a generalised Bott manifold.  
\end{cor}

\section{Torus manifolds revisited}
\label{sec:torus-manif-revis}

The proof of the classification of simply connected non-negatively curved torus manifolds given in \cite{MR3355120} proceeds as follows.
First one uses results of Spindeler \cite{spindeler} to show that a simply connected non-negatively curved torus manifold \(M\) is locally standard and that all faces of \(M/T\) are diffeomorphic to standard discs after smoothing the corners.
Moreover, one knows that the two-dimensional faces have at most four vertices.
Using a combinatorial argument (Proposition 4.5 in the cited paper and Theorem~\ref{sec:torus-manif-revis-1} below) shows that the orbit space of \(M\) is combinatorially equivalent to a product \(\prod_i \Delta^{n_i}\times \prod_i\Sigma^{n_i}\).
From this one can then deduce that \(M/T\) is diffeomorphic to that product and it follows that \(M\) is equivariantly diffeomorphic to a quotient of a product of spheres by a free linear torus action.

The proof of Proposition 4.5 given in \cite{MR3355120} is highly technical, very long and not really enlightening. Therefore we want to give a short proof of that proposition based on the results of our paper in this section.

We start by stating the proposition as the following theorem.
\begin{theorem}
  \label{sec:torus-manif-revis-1}
  Let \(Q\) be the orbit space of a locally standard torus manifold \(M^{2n}\), such that all faces of \(Q\) are acyclic over the integers. Assume that each two-dimensional face of \(Q\) has at most four vertices. Then \(Q\) is combinatorially equivalent to a product \(\prod_i \Delta^{n_i}\times \prod_i\Sigma^{n_i}\).
\end{theorem}
\begin{proof}
  For \(n\leq 3\) this follows as in the proof of Lemma~\ref{sec:gkm-graphs-1}.
  Therefore assume \(n\geq 4\). Then \(M\) is a GKM$_4$-manifold.
  Let \(\Gamma\) be the vertex-edge graph of \(Q\). It is the same as the GKM graph of \(M\). Therefore we have a normal covering \(\tilde\Gamma\rightarrow \Gamma\) by the vertex edge graph \(\tilde\Gamma\) of a product \(\prod_i \Delta^{n_i}\times \prod_i\Sigma^{n_i}\).
  Let \(G\) be the deck transformation group of this covering.
  Since the covering is compatible with the connections on \(\Gamma\) and \(\tilde\Gamma\) and the connections determine the face structure, it suffices to show that the covering is trivial.

  We do this by induction on \(n\) starting with \(n=4\).
  If \(n=4\) then because the covering respects the local face structure, its restriction to any \(3\)-dimensional face of \(\tilde\Gamma\) is a trivial covering.
Moreover, if \(n>4\) the same holds for \((n-1)\)-dimensional faces of \(\tilde\Gamma\) by the induction hypothesis.

Therefore in both cases the action of \(G\) on the set of facets of \(\tilde\Gamma\) is free. Moreover, because the labeling \(\lambda\) of the facets with isotropy groups is invariant under the \(G\)-action, we must have
\begin{equation*}
  F\cap gF=\emptyset
\end{equation*}
for all facets \(F\) and all non-trivial \(g\in G\).

There is only one possibility how this can happen:
\(\tilde\Gamma\) is the vertex-edge graph of \([-1,1]^n\) and \(G=\mathbb{Z}/2\) acts by multiplication with \(-1\) on each factor.
So the intersection of any two facets of \(Q\) has precisely two components and all facets are combinatorially equivalent to cubes \([-1,1]^{n-1}\).
But by Lemma 4.4 of \cite{MR3355120} such an orbit space \(Q\) with only acyclic faces does not exist.
\end{proof}

\begin{remark}
  Using the same proof and Proposition 5.1 (f) of \cite{GGKRW} one can show that the same conclusion as in the above theorem holds, if \(Q\) is the orbit space of a not necessary locally standard rationally elliptic torus orbifold.
\end{remark}

\small

\bibliography{nnc_gkm}{}

\begin{thebibliography}{GGKRW18}

\bibitem[AK14]{AmannKennard}
Manuel Amann and Lee Kennard.
\newblock Topological properties of positively curved manifolds with symmetry.
\newblock {\em Geom. Funct. Anal.}, 24(5):1377--1405, 2014.

\bibitem[All78]{Allday}
Christopher Allday.
\newblock On the rational homotopy of fixed point sets of torus actions.
\newblock {\em Topology}, 17(1):95--100, 1978.

\bibitem[BGH02]{BGH}
Ethan Bolker, Victor Guillemin, and Tara Holm.
\newblock How is a graph like a manifold?
\newblock {\em preprint: arXiv:math/0206103}, 2002.

\bibitem[BP02]{BuchstaberPanov1}
Victor~M. Buchstaber and Taras~E. Panov.
\newblock {\em Torus actions and their applications in topology and
  combinatorics}, volume~24 of {\em University Lecture Series}.
\newblock American Mathematical Society, Providence, RI, 2002.

\bibitem[BP15]{BuchstaberPanov2}
Victor~M. Buchstaber and Taras~E. Panov.
\newblock {\em Toric topology}, volume 204 of {\em Mathematical Surveys and
  Monographs}.
\newblock American Mathematical Society, Providence, RI, 2015.

\bibitem[CMS10]{choi10:_quasit}
Suyoung Choi, Mikiya Masuda, and Dong~Youp Suh.
\newblock Quasitoric manifolds over a product of simplices.
\newblock {\em Osaka J. Math.}, 47:109--129, 2010.

\bibitem[DJ91]{DavisJanuszkiewicz}
Michael~W. Davis and Tadeusz Januszkiewicz.
\newblock Convex polytopes, {C}oxeter orbifolds and torus actions.
\newblock {\em Duke Math. J.}, 62(2):417--451, 1991.

\bibitem[ES19]{MR3897041}
Christine Escher and Catherine Searle.
\newblock Non-negatively curved 6-manifolds with almost maximal symmetry rank.
\newblock {\em J. Geom. Anal.}, 29(1):1002--1017, 2019.

\bibitem[FHT01]{FHT}
Yves F\'{e}lix, Stephen Halperin, and Jean-Claude Thomas.
\newblock {\em Rational homotopy theory}, volume 205 of {\em Graduate Texts in
  Mathematics}.
\newblock Springer-Verlag, New York, 2001.

\bibitem[FR05]{FangRong}
Fuquan Fang and Xiaochun Rong.
\newblock Homeomorphism classification of positively curved manifolds with
  almost maximal symmetry rank.
\newblock {\em Math. Ann.}, 332(1):81--101, 2005.

\bibitem[GGKRW18]{GGKRW}
Fernando Galaz-Garc\'{i}a, Martin Kerin, Marco Radeschi, and Michael Wiemeler.
\newblock Torus orbifolds, slice-maximal torus actions, and rational
  ellipticity.
\newblock {\em Int. Math. Res. Not. IMRN}, 2018(18):5786--5822, 2018.

\bibitem[GGS11]{MR2784821}
Fernando Galaz-Garcia and Catherine Searle.
\newblock Low-dimensional manifolds with non-negative curvature and maximal
  symmetry rank.
\newblock {\em Proc. Amer. Math. Soc.}, 139(7):2559--2564, 2011.

\bibitem[GKM98]{GKM}
Mark Goresky, Robert Kottwitz, and Robert MacPherson.
\newblock Equivariant cohomology, {K}oszul duality, and the localization
  theorem.
\newblock {\em Invent. Math.}, 131(1):25--83, 1998.

\bibitem[GKS20]{MR4088352}
S.~Goette, M.~Kerin, and K.~Shankar.
\newblock Highly connected 7-manifolds and non-negative sectional curvature.
\newblock {\em Ann. of Math. (2)}, 191(3):829--892, 2020.

\bibitem[Gro81]{Gromov}
Michael Gromov.
\newblock Curvature, diameter and {B}etti numbers.
\newblock {\em Comment. Math. Helv.}, 56(2):179--195, 1981.

\bibitem[GS94]{GroveSearle}
Karsten Grove and Catherine Searle.
\newblock Positively curved manifolds with maximal symmetry-rank.
\newblock {\em J. Pure Appl. Algebra}, 91(1-3):137--142, 1994.

\bibitem[GW15]{GoertschesWiemeler}
Oliver Goertsches and Michael Wiemeler.
\newblock Positively curved {GKM}-manifolds.
\newblock {\em Int. Math. Res. Not. IMRN}, 2015(22):12015--12041, 2015.

\bibitem[GZ01]{GuilleminZara}
Victor Guillemin and Catalin Zara.
\newblock 1-skeleta, {B}etti numbers, and equivariant cohomology.
\newblock {\em Duke Math. J.}, 107(2):283--349, 2001.

\bibitem[HK89]{HsiangKleiner}
Wu-Yi Hsiang and Bruce Kleiner.
\newblock On the topology of positively curved {$4$}-manifolds with symmetry.
\newblock {\em J. Differential Geom.}, 29(3):615--621, 1989.

\bibitem[HS90]{MR1062771}
A.~Haefliger and \'{E}. Salem.
\newblock Riemannian foliations on simply connected manifolds and actions of
  tori on orbifolds.
\newblock {\em Illinois J. Math.}, 34(4):706--730, 1990.

\bibitem[HS91]{MR1116630}
Andr\'{e} Haefliger and \'{E}liane Salem.
\newblock Actions of tori on orbifolds.
\newblock {\em Ann. Global Anal. Geom.}, 9(1):37--59, 1991.

\bibitem[KL14]{MR3244330}
Bruce Kleiner and John Lott.
\newblock Geometrization of three-dimensional orbifolds via {R}icci flow.
\newblock {\em Ast\'{e}risque}, (365):101--177, 2014.

\bibitem[Kur19]{Kuroki}
Shintaro Kuroki.
\newblock Upper bounds for the dimension of tori acting on {GKM} manifolds.
\newblock {\em J. Math. Soc. Japan}, 71(2):483--513, 2019.

\bibitem[Kus09]{MR2603265}
Andrei~A. Kustarev.
\newblock Equivariant almost complex structures on quasitoric manifolds.
\newblock {\em Tr. Mat. Inst. Steklova}, 266(Geometriya, Topologiya i
  Matematicheskaya Fizika. II):140--148, 2009.

\bibitem[LT97]{MR1401525}
Eugene Lerman and Susan Tolman.
\newblock Hamiltonian torus actions on symplectic orbifolds and toric
  varieties.
\newblock {\em Trans. Amer. Math. Soc.}, 349(10):4201--4230, 1997.

\bibitem[MMP07]{MaedaMasudaPanov}
Hiroshi Maeda, Mikiya Masuda, and Taras Panov.
\newblock Torus graphs and simplicial posets.
\newblock {\em Adv. Math.}, 212(2):458--483, 2007.

\bibitem[MP06]{MasudaPanov}
Mikiya Masuda and Taras Panov.
\newblock On the cohomology of torus manifolds.
\newblock {\em Osaka J. Math.}, 43(3):711--746, 2006.

\bibitem[PS10]{MR2791564}
Mainak Poddar and Soumen Sarkar.
\newblock On quasitoric orbifolds.
\newblock {\em Osaka J. Math.}, 47(4):1055--1076, 2010.

\bibitem[Seg68]{MR0234452}
Graeme Segal.
\newblock Equivariant {$K$}-theory.
\newblock {\em Inst. Hautes \'{E}tudes Sci. Publ. Math.}, 34:129--151, 1968.

\bibitem[Spi14]{spindeler}
Wolfgang Spindeler.
\newblock {\em $S^1$-actions on 4-manifolds and fixed point homogeneous
  manifolds of nonnegative curvature}.
\newblock PhD thesis, WWU M\"unster, 2014.

\bibitem[SY94]{SearleYang}
Catherine Searle and DaGang Yang.
\newblock On the topology of non-negatively curved simply connected
  {$4$}-manifolds with continuous symmetry.
\newblock {\em Duke Math. J.}, 74(2):547--556, 1994.

\bibitem[Wie15]{MR3355120}
Michael Wiemeler.
\newblock Torus manifolds and non-negative curvature.
\newblock {\em J. Lond. Math. Soc. (2)}, 91(3):667--692, 2015.

\bibitem[Wil03]{Wilking}
Burkhard Wilking.
\newblock Torus actions on manifolds of positive sectional curvature.
\newblock {\em Acta Math.}, 191(2):259--297, 2003.

\bibitem[Yer14]{MR3251232}
Dmytro Yeroshkin.
\newblock {\em Riemannian orbifolds with non-negative curvature}.
\newblock ProQuest LLC, Ann Arbor, MI, 2014.
\newblock Thesis (Ph.D.)--University of Pennsylvania.

\end{thebibliography}
\bibliographystyle{alpha}

\normalsize

~\\Oliver Goertsches\\
Fachbereich Mathematik und Informatik,
Philipps-Universit\"at Marburg\\
Hans-Meerwein-Stra\ss e,
D-35032 Marburg,
Germany\\
\texttt{goertsch@mathematik.uni-marburg.de}

~\\Michael Wiemeler\\Mathematisches Institut, WWU M\"unster\\Einsteinstra\ss e 62, D-48149 M\"unster, Germany\\
\texttt{wiemelerm@uni-muenster.de}
\end{document}